\documentclass[a4paper,reqno,12pt]{amsart}
\usepackage[margin=1.0in]{geometry}
\usepackage{stmaryrd} 
\usepackage{amsthm}
\usepackage{amssymb}
\usepackage{amsmath}
\usepackage{mathtools}
\usepackage{enumitem}
\usepackage{epsfig}
\usepackage{tikz-cd}
\usetikzlibrary{decorations.markings,shapes.geometric,matrix,arrows,positioning}
\usepackage[hidelinks]{hyperref}
\usepackage{microtype}
\usepackage{comment}
\def\MR#1{} 

\makeatletter
\def\SL@eqntext#1{\rlap{\qquad\SL@margintext{#1}}}
\makeatother

\tikzset{anchorbase/.style={baseline={([yshift=-0.5ex]current bounding box.center)}},
  int/.style={thick},
  cross line/.style={preaction={draw=white,line width=6pt,-}},
  wall/.style={thin,double,blue},
  middlearrow/.style={postaction=decorate,decoration={markings,mark=at
    position .55 with {\arrow{stealth};}}},
  middlearrowrev/.style={postaction=decorate,decoration={markings,mark=at
    position .55 with {\arrowreversed{stealth};}}},
  ev/.style={shape=rectangle, draw}
}

\newcommand{\tcoev}{\stackrel{\longleftarrow}{\operatorname{coev}}}
\newcommand{\tev}{\stackrel{\longleftarrow}{\operatorname{ev}}}
\newcommand{\ev}{\stackrel{\longrightarrow}{\operatorname{ev}}}
\newcommand{\coev}{\stackrel{\longrightarrow}{\operatorname{coev}}}
\newcommand{\p}[1]{\ensuremath{\bar {#1}}}

\newcommand{\C}{\mathbb{C}}
\newcommand{\R}{\mathbb{R}}
\newcommand{\Z}{\mathbb{Z}}
\newcommand{\I}{\sqrt{-1}}
\newcommand{\unit}{\mathbb{I}}
\newcommand{\q}{\mathfrak{q}}

\newcommand{\cat}{\mathcal{C}}
\newcommand{\FR}{\mathsf{Z}}
\newcommand{\Gr}{\mathsf{G}}
\newcommand{\SSS}{\mathsf{X}}
\newcommand{\Hom}{\textup{\text{Hom}}}
\newcommand{\End}{\textup{\text{End}}}
\newcommand{\id}{\textup{\text{id}}}
\newcommand{\qd}{\mathsf{d}}
\newcommand{\qdim}{\textup{\text{qdim}}}
\newcommand{\mt}{\mathsf{tr}}
\newcommand{\sch}{\operatorname{sch}}
\newcommand{\Ztwo}{\mathbb{Z} \slash 2 \mathbb{Z}}
\newcommand{\ptr}{\operatorname{ptr}}

\newcommand{\height}{\operatorname{ht}}

\newcommand{\TQFT}{\mathcal{Z}}
\newcommand{\Zhat}{\widehat{Z}}
\newcommand{\op}{\textnormal{op}}
\newcommand{\Span}{\textnormal{Span}}

\newcommand{\g}{\mathfrak{g}}
\newcommand{\Cart}{\mathfrak{h}}
\newcommand{\osp}{\mathfrak{osp}}
\newcommand{\ospn}{\osp(2 \vert 2n)}

\newcommand{\qhospn}{U^{H}_q(\ospn)}
\newcommand{\qhospnres}{\overline{U}^{H}_q(\ospn)}

\newcommand{\sdim}{\operatorname{sdim}}

\newcommand{\qbin}[2]{\left\lbrack\begin{matrix}
    #1\\#2
\end{matrix}\right\rbrack}

\newtheorem{Lem}{Lemma}[section]
\newtheorem{Prop}[Lem]{Proposition}

\newtheorem{Def}[Lem]{Definition}

\theoremstyle{plain}
\newtheorem{Thm}[Lem]{Theorem}

\newtheorem{Cor}[Lem]{Corollary}

\theoremstyle{definition}
\newtheorem{examplex}[Lem]{Example}
\newenvironment{Ex}
  {\pushQED{\qed}\examplex}
  {\popQED\endexamplex}

\newtheorem{Rem}[Lem]{Remark}
  
\newcommand{\epsh}[2]
         {\begin{array}{c} \hspace{-1.3mm}
        \raisebox{-4pt}{\epsfig{figure=#1,height=#2}}
        \hspace{-1.9mm}\end{array}}

\title{Relative modular categories from $\mathfrak{osp}(2 \vert 2n)$}

\author[S. Porter]{Seth Porter}
\address{Department of Mathematics and Statistics \\ Utah State University\\
Logan, Utah 84322 \\ USA}
\email{seth.porter@usu.edu}

\author[M.\,B. Young]{Matthew B. Young}
\address{Department of Mathematics and Statistics \\ Utah State University\\
Logan, Utah 84322 \\ USA}
\email{matthew.young@usu.edu}

\date{\today}


\begin{document}

\begin{abstract}
We construct relative modular categories from the weighted representation theory of the unrolled quantum group of the orthosymplectic Lie superalgebra $\mathfrak{osp}(2 \vert 2n)$.
\end{abstract}

\maketitle

\setcounter{tocdepth}{1}

\tableofcontents

\section*{Introduction}
\addtocontents{toc}{\protect\setcounter{tocdepth}{1}}

A central role in quantum topology is played by the Reshetikhin--Turaev construction, which produces from a modular tensor category $\cat$ a $3$-dimensional topological quantum field theory (TQFT) $\TQFT_{\cat}$ \cite{reshetikhin1991}. Recall that a modular tensor category is a finite semisimple ribbon category which satisfies an additional nondegeneracy (modularity) condition. The main classes of examples of modular tensor categories arise from the representation theory of twisted Drinfeld doubles of finite groups \cite{dijkgraaf1990b} and the semisimplified representation theory of small quantum groups of complex simple Lie algebras at suitable roots of unity \cite{andersen1995}. The resulting TQFTs provide rigorous mathematical models of Dijkgraaf--Witten \cite{dijkgraaf1990} and Chern--Simons quantum field theories \cite{witten1989}, respectively.

Recent years have seen a wealth of activity devoted to extending techniques, constructions and examples in quantum topology beyond the traditional finite semisimple regime, leading to what is now known as non-semisimple quantum topology. Most relevant to the present paper is the construction of De Renzi, building on earlier works of Blanchet, Costantino, Geer and Patureau-Mirand \cite{costantino2014,blanchet2016}, which produces from a relative modular category $\cat$ a $3$-dimensional decorated TQFT $\TQFT_{\cat}$ \cite{derenzi2022}; if $\cat$ is in fact modular, then $\TQFT_{\cat}$ is the Reshetikhin--Turaev TQFT from above. Relative modular categories are non-semisimple, non-finite generalizations of modular tensor categories. The resulting TQFT $\TQFT_{\cat}$ is decorated in the sense that, for example, it assigns numerical invariants to closed oriented $3$-manifolds together with the data of a suitably generic flat connection (with structure group determined by the defining data of $\cat$). TQFTs arising from relative modular categories are strictly stronger than those arising from modular tensor categories. For example, the former are able to distinguish homotopy classes of lens spaces while the latter cannot \cite{blanchet2016}.

At present, all known examples of relative modular categories arise from the representation theory of restricted unrolled quantum groups, specifically for the following complex Lie (super)algebras: simple Lie algebras \cite{derenzi2020}, special linear Lie superalgebras $\mathfrak{sl}(n \vert m)$, $n \neq m$, \cite{anghel2021}, certain Lie superalgebras with abelian bosonic part, such as $\mathfrak{gl}(1 \vert 1)$, \cite{geer2025,garner2025} and the orthosymplectic Lie superalgebra $\osp(1 \vert 2)$ \cite{costantino2026}. The resulting decorated TQFTs have been matched to physical quantum field theories in the case of $\mathfrak{sl}(2)$ and Lie superalgebras with abelian bosonic part, where the relevant theories are the A-type topological twist of $\mathcal{N}=4$ Chern--Simons-matter theory with gauge group $SU(2)$ \cite{creutzig2024} and torus-gauged Rozansky--Witten theories of linear holomorphic symplectic manifolds \cite{geer2025,garner2025}, respectively. It is an open problem to determine the physical origins of $\TQFT_{\cat}$ for more general Lie superalgebras.

In this paper, we construct another family of relative modular categories from Lie superalgebras. To place this family in proper context, recall that a basic classical Lie superalgebra is a simple complex Lie superalgebra $\mathfrak{g}$ whose bosonic subalgebra $\g_{\p 0}$ is reductive and for which $\g$ admits a nondegenerate $\g$-invariant even supersymmetric bilinear form. Basic classical Lie superalgebras, which are in many ways the natural super analogues of complex semisimple Lie algebras, naturally divide into two types according to whether the fermionic summand $\g_{\p 1}$ is a direct sum of two simple $\g_{\p 0}$-modules (Type I) or is a simple $\g_{\p 0}$-module (Type II). Kac classified Type I Lie superalgebras into three\footnote{Despite not being simple, the Lie superalgebra $\mathfrak{gl}(m \vert n)$ is sometimes declared to be basic classical, in which case it is of Type I. With this convention, there is a fourth family of Type I Lie superalgebras.} families \cite{kac1977}:
\begin{enumerate}
\item $\mathfrak{sl}(m \vert n)$, $m > n \geq 1$,
\item $\mathfrak{psl}(n \vert n)$, $n \geq 2$,
\item $\osp(2 \vert 2n)$, $n \geq 2$.
\end{enumerate}
As stated above, relative modular categories have been constructed from the first family \cite{anghel2021}. The second and third families have not been treated. Our main result is the construction of relative modular categories from the third family.

To state our main result, we begin with the unrolled quantum group of $\ospn$:
\[
\qhospnres
:=
\overline{U}_q(\ospn) \rtimes \C[\Cart].
\]
Here $\overline{U}_q(\ospn)$ is the restricted quantum group of $\ospn$ at the root of unity $q=e^{\frac{2 \pi i}{r}}$, wherein $r$\textsuperscript{th} powers of the even root vectors are set to zero, and the group algebra $\cat[\Cart]$ of the Cartan subalgebra acts on $\overline{U}_q(\ospn)$ by commuting with the group-like Cartan generators and by the classical Lie superalgebra relations with the root generators. A detailed construction of $\qhospnres$ as a Hopf superalgebra is presented in Section \ref{sec:qGroupOSP}. A weight $\qhospnres$-module is a $\qhospnres$-module on which $\C[\Cart]$ acts semisimply and such that the action of each group-like Cartan generator is the appropriate $q$-power of the action of the associated Cartan element.

\newtheorem*{ltmp}{Theorem \ref{thm:modularity}}
\begin{ltmp}
Assume that $q = e^{\frac{2 \pi i}{r}}$ for an odd integer $r$ which does not divide $n$, then the category $\cat(r,n)$ of weight $\qhospnres$-modules admits a relative modular structure.
\end{ltmp}

The proof of Theorem \ref{thm:modularity} is spread across Section \ref{sec:repOSP}. Key aspects of the proof include:
\begin{itemize}
\item Theorem \ref{thm:genericSemiSimple}: The category $\cat(r,n)$ is generically semisimple with respect to its canonical grading by $\Cart^{\vee}$ modulo the root lattice. Generic simple modules are those Kac modules which are induced from simple Verma $\overline{U}_q^H(\ospn_{\p 0})$-modules.
\item Corollary \ref{cor:braidedCat}: The category $\cat(r,n)$ inherits a braiding by truncating Yamane's universal $R$-matrix for the $h$-adic quantum group $U_h(\ospn)$ \cite{yamane1994}.
\item Proposition \ref{prop:unimodularity}: The category $\cat(r,n)$ is unimodular. In particular, $\cat(r,n)$ admits a nondegenerate modified trace on its projective ideal.
\item Theorem \ref{thm:modularity}: The category $\cat(r,n)$ admits a relative modularity parameter.
\end{itemize}
In the process of proving Theorem \ref{thm:modularity}, we provide explicit formulae for ribbon twists (Proposition \ref{prop:ribCat}) and modified quantum dimensions (Corollary \ref{cor:mtraceNormalized}) of generic Kac modules. We also explicitly compute the stabilization coefficients $\Delta_{\pm}$ of $\cat(r,n)$ (Section \ref{sec:stabCoeff}).

Theorem \ref{thm:modularity} together with De Renzi's construction produces a $3$-dimensional decorated TQFT $\TQFT_{\cat(r,n)}$. The explicit control of the representation theory of $\qhospnres$, as detailed in the previous paragraph, makes $\TQFT_{\cat(r,n)}$ ripe for calculations. For example, the value of $\TQFT_{\cat(r,n)}$ on plumbed $3$-manifolds and circle bundles over closed surfaces can be computed as in \cite{costantino2026}.

It is an interesting problem to extend the construction of relative modular categories to Type II Lie superalgebras, such as $\osp(m \vert 2n)$, $m \neq 2$, and $D(2 \vert 1 ; \alpha)$, $\alpha \in \C \setminus \{0,-1\}$. At present, only the case $\osp(1 \vert 2)$ has been treated \cite{costantino2026}; it is likely that a similar direct approach will extend to $\osp(1 \vert 2n)$. The family $\osp(1 \vert 2n)$ is rather special among Lie superalgebras, being the only family whose classical representation theory is semisimple \cite{djokovic1976}. A key issue in the Type II case is that Kac modules are, in general, not simple. Rather, the generic simple modules have dimension strictly smaller than that of a Kac module, spoiling the strategy of proof of generic semisimplicity given in Theorem \ref{thm:genericSemiSimple}. 

We close by offering an additional piece of motivation for considering quantum invariants associated with Lie superalgebras. With a view toward categorifying $\mathfrak{sl}(2)$ Reshetikhin--Turaev invariants, Gukov, Pei, Putrov and Vafa recently introduced a new $\mathbf{q}$-series invariant of (certain classes of) $3$-manifolds, known as $\Zhat^{\mathfrak{sl}(2)}$-invariants \cite{gukov2020}. More generally, invariants $\Zhat^{\g}$ associated with a simple Lie algebra $\g$ were introduced by Park \cite{park2020} and their relation to Reshetikhin--Turaev invariants established by Murakami and Terashima \cite{murakami2024,murakami2026}. Finally, Ferrari and Putrov introduced $\Zhat^{\g}$-invariants for certain Lie superalgebras $\g$ \cite{ferrari2024}. Since Reshetikhin--Turaev invariants do not exist for Lie superalgebras, it is natural to seek a general relation between $\Zhat^{\g}$-invariants and the $3$-manifold invariants defined by a relative modular category associated with $\g$. Such a relation has been established for $\mathfrak{sl}(2)$ \cite{costantino2023}, $\mathfrak{sl}(2 \vert 1)$ \cite{ferrari2024} and $\osp(1 \vert 2)$ \cite{costantino2026}. The results of this paper establish the existence of the latter invariants for $\ospn$, which can then be compared with the $\Zhat^{\ospn}$-invariants explored by Chae \cite{chae2021}.

\subsection*{Acknowledgments}
The authors thank Francesco Costantino, Nathan Geer, Matthew Harper, Bertrand Patureau-Mirand and Adam Robertson for discussions. M.\,Y. is partially supported by National Science Foundation grants DMS-2302363 and DMS-2440471 and Simons Foundation Collaboration Grant for Mathematicians (Award ID 853541).  M.\, Y. thanks Universit\'{e} Toulouse III Paul Sabatier, where part of this research was completed, for support and hospitality.

\section{Reminders on relative modular categories}
\addtocontents{toc}{\protect\setcounter{tocdepth}{1}}
\label{sec:background}

\subsection{Ribbon categories}
\label{sec:monCat}

The reader is referred to \cite{etingof2015} for detailed background on monoidal categories.

Let $\cat$ be a $\C$-linear abelian monoidal category. We assume that the tensor functor $\otimes: \cat \times \cat \rightarrow \cat$ is $\C$-bilinear and the monoidal unit $\unit \in \cat$ is simple. If $\cat$ is rigid, braided and a compatible twist, then $\cat$ is a \emph{ribbon category}. Basic morphisms in $\cat$ are depicted as follows:
\[
\id_V
=
\begin{tikzpicture}[anchorbase]
\draw[->,thick] (0,0) -- node[left] {\small$V$} (0,1);
\end{tikzpicture}
\qquad ,\qquad
\id_{V^{\vee}}
=
\begin{tikzpicture}[anchorbase]
\draw[<-,thick] (0,0) -- node[left] {\small$V$} (0,1);
\end{tikzpicture}
\]
\[
\tev_V
=
\begin{tikzpicture}[anchorbase]
\draw[->,thick] (0,0)  arc (0:180:0.5 and 0.75);
\node at (-1.4,0)  {$V$};
\end{tikzpicture}
\qquad, \qquad
\tcoev_V
=
\begin{tikzpicture}[anchorbase]
\draw[<-,thick] (0,0)  arc (180:360:0.5 and 0.75);
\node at (-0.5,-0.1)  {$V$};
\end{tikzpicture}
\]
\[
\ev_V
=
\begin{tikzpicture}[anchorbase]
\draw[<-,thick] (0,0)  arc (0:180:0.5 and 0.75);
\node at (0.4,0)  {$V$};
\end{tikzpicture}
\qquad, \qquad
\coev_V
=
\begin{tikzpicture}[anchorbase]
\draw[->,thick] (0,0)  arc (180:360:0.5 and 0.75);
\node at (1.5,-0.1)  {$V$};
\end{tikzpicture}
\]
\[
c_{V,W}
=
\begin{tikzpicture}[anchorbase]
\draw[->,thick] (0.5,0) -- node[right,near start] {\small $W$} (0,1);
\draw[->,thick,cross line] (0,0) -- node[left,near start] {\small$V$} (0.5,1);
\end{tikzpicture}
\qquad, \qquad
\theta_V
=
\begin{tikzpicture}[anchorbase]
\draw[->,thick,rounded corners=8pt] (0.25,0.25) -- (0,0.5) -- (0,1);
\draw[thick,rounded corners=8pt,cross line] (0,0) -- (0,0.5) -- (0.25,0.75);
\draw[thick] (0.25,0.75) to [out=30,in=330] (0.25,0.25);
\node at (-0.2,0.2)  {$V$};
\end{tikzpicture}.
\]
We read diagrams left to right and bottom to top. For example, the braiding and twists are morphisms $c_{V,W}: V \otimes W \rightarrow W \otimes V$ and $\theta_V: V \rightarrow V$, respectively.

\subsection{Relative modular categories}
\label{sec:relModCat}

The reader is referred to \cite{derenzi2022} for details on relative modular categories.

Let $\cat$ be a $\C$-linear abelian monoidal category. Recall that an \emph{ideal} of $\cat$ is a full subcategory which is closed under retracts and tensor product (on either side) by objects of $\cat$. The full subcategory of projective objects of $\cat$ is an example of an ideal. Given $V , W \in \cat$, the right partial trace along $W$ is the map $\ptr_W : \End_{\cat}(V \otimes W) \rightarrow \End_{\cat}(V)$ defined by
\[
\ptr_W(f)
=
(\id_V \otimes \ev_W) \circ (f \otimes \id_{W^{\vee}}) \circ (\id_V \otimes \tcoev_W).
\]

\begin{Def}
\label{def:mtrace}
\begin{enumerate}
\item A \emph{modified trace} on an ideal $\mathcal{I} \subset \cat$ is the data of $\C$-linear functions 
$\mt=\{\mt_V:\End_\cat(V) \rightarrow \C \}_{V \in \mathcal{I}}$ which satisfy:
\begin{enumerate}
\item $\mt_V(f \circ g)=\mt_W(g \circ f)$ for all $f: W \rightarrow V$ and $g: V \rightarrow W$ in $\mathcal{I}$.
\item $\mt_{V\otimes W}(f)=\mt_V(\ptr_W(f))$ for all $V \in \mathcal{I}$, $W \in \cat$ and $f\in\End_{\cat}(V\otimes W)$.
\end{enumerate}

\item The \emph{modified dimension} of $V\in \mathcal{I}$ is $\qd(V)=\mt_V(\id_V)$.
\end{enumerate}
\end{Def}

\begin{Def}
\begin{enumerate}
\item A set $\mathcal{E}=\{ V_j \mid j \in J \}$ of objects of $\cat$ is \emph{dominating} if for any $V \in \cat$ there exist indices $\{j_1,\dots,j_m \} \subset J $ and morphisms $\iota_k \in \Hom_{\cat}(V_{j_k},V)$ and $s_k \in \Hom_{\cat}(V,V_{j_k})$ such that $\id_{V}=\sum_{k=1}^m \iota_k \circ s_k$.

\item A dominating set $\mathcal{E}$ is \emph{completely reduced} if $\dim_\C \Hom_{\cat}(V_j,V_k)=\delta_{jk}$ for all $j,k \in J$.
\end{enumerate}
\end{Def}

\begin{Def}
\label{def:freeReal}
A \emph{free realization} of an abelian group $\FR$ on a ribbon category $\cat$ is a monoidal functor $\sigma: \FR \rightarrow \cat$ such that
\begin{enumerate}
\item $\sigma(0) = \unit$,
\item $\theta_{\sigma(z)}=\id_{\sigma(z)}$ for all $z \in \FR$, and
\item if $V \otimes \sigma(z) \simeq V$ for a simple object $V \in \cat$, then $z=0$.
\end{enumerate}
\end{Def}

We often identify the functor $\sigma: \FR \rightarrow \cat$ with the set of objects $\sigma(\FR) := \{\sigma(z) \mid z \in \FR\}$.

\begin{Def}
\label{def:Gstr}
Let $(\Gr,+)$ be an abelian group.
\begin{enumerate}
\item A subset $\SSS \subset \Gr$ is \emph{symmetric} if $\SSS=-\SSS$ and \emph{small} if $\bigcup_{i=1}^s (g_i+\SSS) \neq \Gr$ for all $g_1,\ldots ,g_s\in \Gr$.

\item A \emph{$\Gr$-grading} of a ribbon category $\cat$ is a decomposition $\cat \simeq \bigoplus_{g \in \Gr} \cat_g$ into full abelian subcategories such that
\begin{enumerate}
\item $\mathbb{I} \in \cat_0$,
\item if $V\in\cat_g$, then $V^{\vee}\in\cat_{-g}$, and
\item if $V\in\cat_g$ and $V^{\prime} \in \cat_{g^{\prime}}$, then $V\otimes V^{\prime}\in\cat_{g+g^{\prime}}$.
\item A $\Gr$-graded category $\cat$ is \emph{generically semisimple} if there exists a small symmetric subset $\SSS \subset \Gr$ such that $\cat_g$ is semisimple whenever $g \in \Gr \setminus \SSS$.
\end{enumerate}
\end{enumerate}
\end{Def}

\begin{Def}
\label{def:preMod}
Let $\Gr$ and $\FR$ be abelian groups, $\SSS \subset \Gr$ a small symmetric subset and $\cat$ a $\Gr$-graded ribbon category with free realization $\sigma : \FR \rightarrow \cat_0$ and non-zero modified trace $\mt$ on its projective ideal. Then $\cat$ is a \emph{pre-modular $\Gr$-category relative to $(\FR,\SSS)$} if it has the following properties:
\begin{enumerate}
\item \emph{Generic semisimplicity}: \label{def:genSS} For each $g \in \Gr \setminus \SSS$, there exists a finite set of simple objects $\Theta(g):=\{ V_i \mid i \in I_g  \}$ such that
\[
\Theta(g) \otimes \sigma(\FR):=\{ V_i \otimes \sigma(z) \mid i \in I_g, \; z \in \FR \}
\]
is a completely reduced dominating set of $\cat_g$.

\item
\label{def:compat}
There exists a bicharacter $\psi: \Gr \times \FR \rightarrow \C^{\times}$ such that
  \begin{equation*}
    \label{eq:psi}
    c_{\sigma(z),V}\circ c_{V,\sigma(z)}= \psi(g,z) \cdot  \id_{V \otimes \sigma(z)}
  \end{equation*}
for all $g\in \Gr$, $V \in \cat_g$ and $z \in \FR$.
\end{enumerate}
\end{Def}
 
\begin{Def}
\label{def:ndeg}
Let $\cat$ be a pre-modular $\Gr$-category relative to $(\FR,\SSS)$.
\begin{enumerate}
\item The \emph{Kirby color of index $g \in \Gr \setminus \SSS$} is $\Omega_g:= \sum_{V \in \Theta(g)}\qd(V) \cdot V$.
\item The \emph{stabilization coefficients} $\Delta_\pm\in \C$ are defined by the skein equivalences
\[
\epsh{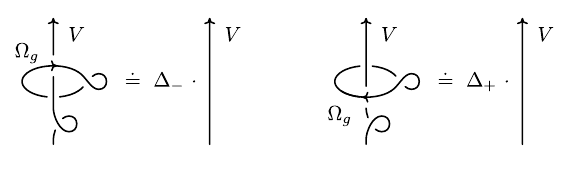}{15ex}
\]
for any $g \in \Gr \setminus \SSS$ and $V\in \cat_g$.
\end{enumerate}
\end{Def}

\begin{Def}
\label{def:modG}
A \emph{modular $\Gr$-category relative to $(\FR,\SSS)$} is a pre-modular category $\mathcal{C}$ for which there exists a scalar $\zeta \in \C^{\times}$, the \emph{relative modularity parameter}, which satisfies
\begin{equation}\label{eq:mod}
    \epsh{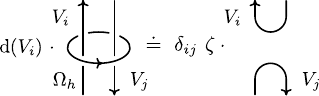}{10ex}
\end{equation}
for all $g,h \in \Gr \setminus \SSS$ and $i,j \in I_g$.
\end{Def}

\section{Quantum groups associated with $\ospn$}
\label{sec:qGroupOSP}

\subsection{Conventions}

Unless mentioned otherwise, the ground field is $\C$. Write $\Ztwo$ for the commutative ring $\{\p 0, \p 1\}$.

We follow the superalgebra conventions of \cite[I-Supersymmetry]{deligne1999}. The adjective `super' is omitted when it will not cause confusion. The parity of a homogeneous vector $v$ of a super vector space $V=V_{\p 0} \oplus V_{\p 1}$ is $\p v \in \Ztwo$. Morphisms of super vector spaces are parity-preserving linear maps. The tensor product of super vector spaces $V$ and $W$ is the tensor product of their underlying vector spaces with $\Ztwo$-grading
\[
(V \otimes W)_{\p p} = \bigoplus_{\p s + \p t = \p p} V_{\p s} \otimes W_{\p t}.
\]

Let $A$ be a superalgebra. The super commutator of homogeneous elements $a,b \in A$ is $[a,b] = ab -(-1)^{\p a \p b} ba$. A (left) super $A$-module is a super vector space $M$ with an $A$-module structure whose action map $A \otimes M \rightarrow M$ is a morphism of super vector spaces. When $\End_A(M) = \C \cdot \id_M$, the scalar by which $f \in \End_A(M)$ acts is denoted $\langle f \rangle \in \C$.

\subsection{The orthosymplectic Lie superalgebra}
\label{sec:ospn}

Let $n$ be a positive integer. Let $\ospn$ be the  orthosymplectic Lie superalgebra of type $\textnormal{C}(n+1)$ in Kac's classification \cite[\S 2.1.2]{kac1977}. Concretely, consider the supersymmetric bilinear form on $\C^{2 \vert 2n}$ determined by the block diagonal matrix $\left(
\begin{smallmatrix}
J^{\textnormal{symm}} & 0 \\
0 & J^{\textnormal{skew}}
\end{smallmatrix}
\right)
$,
where
\[
J^{\textnormal{symm}}
=
\left(
\begin{matrix}
0 & 1 \\
1 & 0
\end{matrix}
\right),
\qquad
J^{\textnormal{skew}}
=
\left(
\begin{matrix}
0 & \mathbf{1}_n \\
-\mathbf{1}_n & 0
\end{matrix}
\right)
\]
and $\mathbf{1}_n$ is the $n \times n$ identity matrix. Then $\ospn$ is the Lie subalgebra of block matrices $\left( \begin{smallmatrix} A & B \\ C & D \end{smallmatrix} \right) \in \mathfrak{gl}(\C^{2 \vert 2n})$ which satisfy 
\[
A^t J^{\textnormal{symm}} + J^{\textnormal{symm}} A = 0,
\qquad
B^t J^{\textnormal{symm}} - J^{\textnormal{skew}} C = 0,
\qquad
D^t J^{\textnormal{skew}} + J^{\textnormal{skew}} D = 0.
\]
The even Lie subalgebra of $\ospn$ is $\ospn_{\p 0} \simeq \mathfrak{so}(2) \oplus \mathfrak{sp}(2n)$.

We take for the Cartan subalgebra $\Cart$ the subspace of all diagonal matrices in $\ospn$. The dual space $\Cart^{\vee} = \Hom_{\C}(\Cart,\C)$ has basis $\{\delta_0, \dots, \delta_n\}$, where $\delta_0$ (resp. $\delta_i$, $1 \leq i \leq n$) records the first (resp. $(2+i)$\textsuperscript{th}) diagonal entry of an element of $\Cart$. With this notation, the positive roots $\Delta^+ = \Delta^+_{\p 0} \sqcup \Delta^+_{\p 1}$ are
\[
\Delta_{\p 0}^+ =
\{\delta_i \pm \delta_j \mid 1 \leq i < j \leq n\} \sqcup \{2 \delta_i \mid 1 \leq i \leq n\},
\]
\[
\Delta_{\p 1}^+ =
\{\delta_0 \pm \delta_i \mid 1 \leq i \leq n\}
\]
and the simple positive roots $\Pi^+ = \Pi^+_{\p 0} \sqcup \Pi^+_{\p 1}$ are
\[
\Pi_{\p 0}^+ = \{\alpha_i := \delta_i-\delta_{i+1} \mid 1 \leq i \leq n-1\} \sqcup \{\alpha_n := 2 \delta_n\},
\]
\[
\Pi_{\p 1}^+ = \{\alpha_0 := \delta_0 - \delta_1\}.
\]
The half-sums of the positive even and odd roots are 
\[
\rho_{\p 0}
=
\frac{1}{2} \sum_{\alpha \in \Delta_{\p{0}}^+} \alpha
= \sum_{i=1}^n (n-i+1)\delta_i,
\qquad
\rho_{\p 1}
=
\frac{1}{2} \sum_{\alpha \in \Delta_{\p 1}^+} \alpha
= n \delta_0
\]
so that the super Weyl vector is
\[
\rho := \rho_{\p0} - \rho_{\p 1} = -n \delta_0 + \sum_{i=1}^n (n-i+1)\delta_i.
\]
Write $\epsilon(\alpha) = (-1)^{\p \alpha} \in \{\pm 1\}$ for the (multiplicative) parity of a root $\alpha \in \Delta$.

Let $\ospn_{\pm 1} \subset \ospn$ be the subspace spanned by the root vectors associated with $\Delta^{\pm}_{\p 1}$. Writing $\ospn_0 = \ospn_{\p 0}$, there is a decomposition of $\ospn_{\p 0}$-modules
\[
\ospn = \ospn_{-1} \oplus \ospn_0 \oplus \ospn_{1}
\]
whose first and third summands are simple. Such a compatible $\Z$-grading is a characteristic property of Type I basic classical Lie superalgebras \cite[\S 2.2]{kac1977}.

Define a bilinear form $(-,-): \Cart^{\vee} \times \Cart^{\vee} \rightarrow \C$ by $(\delta_i,\delta_j) = -(-1)^{\delta_{0i}} \delta_{ij}$. The Cartan matrix $(a_{ij})$ has entries
\[
a_{ij}
=
\begin{cases}
\frac{2 (\alpha_i,\alpha_j)}{(\alpha_i,\alpha_i)} & \mbox{if } i \neq 0, \\
(\alpha_0,\alpha_j) & \mbox{if } i = 0.
\end{cases}
\]
Setting $d_0=-1$, $d_1 = \cdots = d_{n-1}=1$ and $d_n = 2$ gives the symmetrized Cartan matrix
\[
(d_i a_{ij})
=
\left( \begin{smallmatrix}
0 & -1 & & & & \\
-1 & 2 & -1 & & & \\
& \ddots & \ddots & \ddots & &\\
& & -1 & 2 & -1 & \\
& & & -1 & 2 & -2 \\
& & & & -2 & 4
\end{smallmatrix} \right).
\]
Let $\langle-,-\rangle: \Cart^{\vee} \times \Cart^{\vee} \rightarrow \C$ be the symmetric bilinear form defined by $\langle \alpha_i, \alpha_j \rangle = d_ia_{ij}$.

\begin{Lem}
\label{lem:rhoPairings}
The following statements hold.
\begin{enumerate}
\item \label{ite:evRhoSimp} $\langle 2\rho, \alpha_i \rangle = \langle \alpha_i, \alpha_i \rangle$ for all $\alpha_i \in \Pi^+$.

\item \label{ite:oddRhoSimp} $\langle 2\rho_{\p 1},\alpha_i \rangle = -2 n \delta_{i0}$ for all $\alpha_i \in \Pi^+$.

\item \label{ite:oddRhoOdd} $\langle 2\rho_{\p 1},\alpha \rangle = -2 n$ for all $\alpha \in \Delta_{\p 1}^+$.

\item \label{ite:eveRhoOdd} $\langle \rho_{\p 0}, \alpha \rangle \in \Z$ for all $\alpha \in \Delta_{\p 1}^+$.
\end{enumerate} 
\end{Lem}

\begin{proof}
The first two equalities are direct computations and the second implies the third. For the final statement, we use the equality $\rho_{\p 0} = \sum_{i=1}^{n-1}\frac{i(2n-i+1)}{2}\alpha_i + \frac{n(n+1)}{2}\alpha_n$ to compute
\begin{equation}
\label{eq:eveRhoOdd}
\langle \rho_{\p 0}, \alpha \rangle
=
\sum_{i=1}^{n-1} \frac{i(2n-i+1)}{2} \langle \alpha_i , \alpha \rangle + \frac{n(n+1)}{2} \langle \alpha_n, \alpha \rangle.
\end{equation}
Writing $\alpha = \sum_{j = 0}^n c_j \alpha_j$ with $c_j \in \Z_{\geq 0}$ and $c_0 = 1$ and noting that $\frac{i(2n-i+1)}{2} \in \Z$ shows that each of the first $n-1$ summands on the right-hand side of equation \eqref{eq:eveRhoOdd} is an integer. Finally, $\frac{n(n+1)}{2} \langle \alpha_n, \alpha\rangle = d_n \frac{n(n+1)}{2}\sum_{j = 0}^n c_j a_{nj}$ is an integer since $d_2 \frac{n(n+1)}{4} \in \Z$.
\end{proof}

The fundamental dominant weights $w_0, \dots, w_n \in \Cart^{\vee}$ are defined so that $\langle w_i, \alpha_j \rangle = d_i \delta_{ij}$. The root lattice $\Lambda_R$ (resp. weight lattice $\Lambda_W$) is the additive subgroup of $\Cart^{\vee}$ generated by $\Delta$ (resp. the fundamental dominant weights). Note that $\Lambda_R \subset \Lambda_W$.

\subsection{The quantum group}
\label{sec:quantGroup}

We define the quantum group of $\ospn$ following \cite[\S 2.1]{zhang1993} and \cite[\S 2.1]{yamane1994}.

Let $r \geq 3$ be an odd integer which does not divide $n$. Set $q = e^{\frac{2 \pi \I}{r}}$. For $z \in \C$, define
\[
q^z = e^{\frac{2 z \pi \I}{r}},
\qquad
\{z \}_q = q^z - q^{-z},
\qquad
\left[ z \right]_q = \frac{\{z\}_q}{\{1\}_q}.
\]
Note that $\{z\}_q =0$ if and only if $z \in \frac{r}{2} \Z$. For integers $0 \leq k \leq m$, set
\[
\left[m \right]_{q}! = \prod_{j=1}^m \left[j\right]_{q},
\qquad
\qbin{m}{k}_{q} = \frac{\left[m \right]_{q}!}{\left[ k \right]_{q}! \left[m-k\right]_{q}!}.
\]
Set $q_i = q^{d_i}$ for $0 \leq i \leq n$.

\begin{Def}
\label{def:quantGroup}
Let $U_q(\ospn)$ be the unital superalgebra with generators $K_i^{\pm 1}$, $E_i$, $F_i$, $0 \leq i \leq n$, all even except for the odd generators $E_0$ and $F_0$, and relations
\[
[K_i,K_j]=0,
\qquad
K_i K_i^{-1} = K_i^{-1} K_i = 1,
\]
\[
K_i E_j = q_i^{a_{ij}} E_j K_i,
\qquad
K_i F_j = q_i^{-a_{ij}} F_j K_i,
\qquad
[E_i,F_j]= \delta_{ij} \frac{K_i -K_i^{-1}}{q_i-q_i^{-1}},
\]
\[
\sum_{t=0}^{1-a_{ij}} (-1)^t \qbin{1 - a_{ij}}{t}_{q_i} E_i^{1-a_{ij}-t} E_j E_i^t =0, \qquad 0 \neq i \neq j,
\]
\[
\sum_{t=0}^{1-a_{ij}} (-1)^t \qbin{1 - a_{ij}}{t}_{q_i} F_i^{1-a_{ij}-t} F_j F_i^t =0, \qquad 0 \neq i \neq j,
\]
\[
E_0^2=0,
\qquad
F_0^2=0.
\]
\end{Def}

The superalgebra $U_q(\ospn)$ has a Hopf structure with coproduct
\[
\Delta(K_i^{ \pm 1}) = K_i^{\pm 1} \otimes K_i^{\pm 1},
\qquad
\Delta(E_i) = E_i \otimes 1 + K_i \otimes E_i,
\qquad
\Delta(F_i) = F_i \otimes K_i^{-1} + 1 \otimes F_i,
\]
counit
\[
\epsilon(K_i^{\pm 1}) = 1,
\qquad
\epsilon(E_i) = 0, \qquad \epsilon(F_i) = 0,
\]
and antipode
\[
\qquad S(K_i^{\pm 1}) = K_i^{\mp 1},
\qquad
S(E_i) = -K_i^{-1}E_i,
\qquad S(F_i) = -F_iK_i.
\]
The even elements $K_i^{\pm 1}, E_i,F_i$, $1 \leq i \leq n$, generate a Hopf subalgebra isomorphic to $U_q(\mathfrak{sp}(2n))$; including the elements $K_0^{\pm 1}$ generates a copy of $U_q(\ospn_{\p{0}})$.

\begin{Rem}
\label{rem:compareHopfStr}
The Hopf structure of $U_q(\ospn)$ used in this paper agrees with \cite{zhang1993,yamane1994} but differs from \cite{khoroshkin1991,anghel2021}. 
\end{Rem}

Let $\omega$ be the $\C$-antilinear antiautomorphism of $U_q(\ospn)$ defined on generators by
\[
\omega(K_i^{\pm 1})=K^{\mp 1}_i,
\qquad
\omega(E_i)=F_i,
\qquad
\omega(F_i)=E_i.
\]

We recall a construction of root vectors for $U_q(\ospn)$. Let $\alpha \in \Delta^+$, written as $\sum_{i=0}^n c_i\alpha_i$ with $c_i \in \Z_{\geq 0}$. Define positive integers
\[
\height(\alpha) = \sum_{i=0}^n c_i,
\qquad
g(\alpha) = \min\{i \mid c_i \neq 0\}.
\]
Set $\height^{\prime}(\alpha) = \frac{\height(\alpha)}{c_{g(\alpha)}}$. Following \cite[\S 3.2]{yamane1994}, define a total order $<$ on  $\Delta^+$ by setting $\alpha < \beta$ if $g(\alpha) < g(\beta)$ or $g(\alpha) = g(\beta)$ and $\height^{\prime}(\alpha) < \height^{\prime}(\beta)$.

\begin{Def}[{\cite[\S 5.1.1]{yamane1994}}]
\label{def:rootVect}
For each $\alpha \in \Delta^+$, define $E_{\alpha} \in U_q(\ospn)$ as follows. Define $E_{\alpha_i} = E_i$ for $0 \leq i \leq n$. If $g(\alpha) < i$ for some $0 \leq i \leq n$ and $\alpha + \alpha_i \in \Delta^+$, then define
 \[
E_{\alpha + \alpha_i} = E_{\alpha}E_{\alpha_i} - q^{-\langle \alpha,\alpha_i \rangle}E_{\alpha_i}E_{\alpha}.
\]
If $g(\alpha) = g(\beta)$, $\height(\beta) - \height(\alpha) \leq 1$ and $\alpha+ \beta \in \Delta^+$, then define 
\[
E_{\alpha + \beta} = \frac{1}{[2]_q} \left( E_{\alpha}E_{\beta} - q^{-\langle \alpha,\beta \rangle} E_{\beta} E_{\alpha} \right).
\]
The negative root vector $F_{\alpha}$ is defined to be $\omega(E_{\alpha})$.
\end{Def}

We note that the odd root vectors $E_{\alpha}$, $F_{\alpha}$, $\alpha \in \Delta^+_{\p 1}$, differ from those of \cite[\S 2.2]{zhang1993} by a non-zero scalar.

\subsection{Unrolled quantum groups}
\label{sec:unrolledQuantGroup}

We keep the notation of Section \ref{sec:quantGroup}.

\begin{Def}
The \emph{unrolled quantum group} $\qhospn$ is the unital superalgebra with generators $H_i$, $K_i^{\pm 1}$, $E_i$, $F_i$, $0 \leq i \leq n$, all even except for the odd generators $E_0$ and $F_0$, and the relations of $U_q(\ospn)$ together with
\[
[H_i,H_j] =0, 
\qquad
[H_i,K_j^{\pm 1}]=0,
\qquad
[H_i,E_j] = a_{ij} E_j,
\qquad
[H_i,F_j] = -a_{ij} F_j.
\]
\end{Def}

Extend the Hopf structure of $U_q(\ospn)$ to $\qhospn$ by setting
\[
\Delta(H_i) = H_i \otimes 1 + 1 \otimes H_i, \qquad S(H_i) = -H_i, \qquad \epsilon(H_i) = 0.
\]

\begin{Lem}
\label{lem:KPowerCentral}
For each $0 \leq i \leq n$, the element $K_i^r$ is central in $\qhospnres$. 
\end{Lem}

\begin{proof}
It suffices to prove that $K_i^r$ commutes with the generators. The defining relations give $K^r_i E_j = q_i^{r a_{ij}} E_j K^r_i$ with $q_i^{ra_{ij}} = q^{rd_ia_{ij}} = 1$. Similarly, $K_i^r$ commutes with $F_j$. It is immediate that $K_i^r$ commutes with $K^{\pm 1}_j$ and $H_j$. 
\end{proof}

Consider the two-sided ideal $I = \langle E_{\alpha}^r, F_{\alpha}^r \mid \alpha \in \Delta_{\p 0}^+\rangle$ of $\qhospn$. Since $r$ is odd and $E_{\alpha}, F_{\alpha} \in U_q^H(\ospn_{\p 0})$, the results of \cite[\S 5.6]{deconcini1992} ensure that $I$ is a Hopf ideal.

\begin{Def}
    The \emph{restricted unrolled quantum group} is the Hopf superalgebra
    \[
     \qhospnres = \qhospn \slash I.
    \]
\end{Def}

Set $m_{\p 0} = r-1$ and $m_{\p 1} = 1$. For $\p p \in \Ztwo$, define even elements of $\qhospnres$ by
\[
\Gamma^+_{\p p} = \prod_{\alpha \in \Delta^+_{\p p}} E^{m_{\p p}}_{\alpha},
\qquad
\Gamma^-_{\p p} = \prod_{\alpha \in \Delta^+_{\p p}} F^{m_{\p p}}_{\alpha},
\]
where the products are taken in the decreasing order with respect to $<$.

\begin{Lem}[{\cite[Lem. 6]{zhang1993}}]
\label{lem:gammaCentral}
The elements $\Gamma^{\pm}_{\p 1}$ commute with the subalgebra $\overline{U}_q(\mathfrak{sp}(2n))$. 
\end{Lem}

\section{The relative modular category of weight modules over $\qhospnres$}
\label{sec:repOSP}

\subsection{Basic definitions}

Let $V$ be a finite-dimensional $\qhospnres$-module. A vector $v \in V$ has \emph{weight} $\lambda \in \Cart^{\vee}$ if $H_i v = \lambda(H_i) v$ for all $0 \leq i \leq n$. We call $V$ a \emph{weight module} if each generator $H_i$ acts semisimply on $V$ and $K_i v = q_i^{\lambda(H_i)} v$ whenever $v \in V$ has weight $\lambda$. Note that $q_i^{\lambda(H_i)} = q^{\langle \lambda, \alpha_i \rangle}$. Let $\cat$ be the category of weight $\qhospnres$-modules and their parity-preserving $\qhospnres$-linear maps. The bialgebra structure of $\qhospnres$ gives $\cat$ a monoidal structure.

\begin{Ex}
\label{ex:oneDimModules}
Let $V \in \cat$ be one-dimensional and $v \in V$ non-zero of weight $\lambda$. Weight considerations imply that $E_i v = F_i v =0 $, and hence $\frac{K_i - K_i^{-1}}{q_i-q_i^{-1}}v=0$ for all $0 \leq i \leq n$ from which we deduce
$2 \langle \lambda, \alpha_i \rangle \in r \Z$. Define
\[
\Lambda^{(1)} = \{\lambda \in \Cart^{\vee} \mid 2\langle \lambda, \alpha_i \rangle \in r\Z \text{ for all } 0 \leq i \leq n \}.
\]
The assignment of $(\lambda, \p{p}) \in \Lambda^{(1)} \times \Ztwo$ to the one-dimensional module $\sigma(\lambda, \p{p})$ of weight $\lambda$ in degree $\p{p}$ parametrizes,isomorphism classes of one-dimensional weight modules.
\end{Ex}

\begin{Lem}
\label{lem:alternativeZ}
    There is an equality $\Lambda^{(1)}  \cap \Lambda_R = r\Lambda_{W} \cap \Lambda_R$.
\end{Lem}

\begin{proof}
Let $\lambda \in \Lambda_R$. Since $\langle \lambda, \alpha_i\rangle \in \Z$ for all $0 \leq i \leq n$ and $r$ is odd, we have $2 \langle \lambda, \alpha_i \rangle \in r \Z$ if and only if $\langle \lambda, \alpha_i \rangle \in r \Z$, whence
\begin{equation}
\label{eq:freeRealWeights}
\Lambda^{(1)}  \cap \Lambda_R = \{ \lambda \in \Lambda_R \mid \langle \lambda, \alpha_i \rangle \in r \Z \mbox{ for all } 0 \leq i \leq n\}.
\end{equation}
On the other hand, the definition of fundamental dominant weights gives
\[
r \Lambda_W \cap \Lambda_R
=
\{ \lambda \in \Lambda_R \mid \langle \lambda, \alpha_i \rangle \in r d_i \Z \mbox{ for all } 0 \leq i \leq n \}.
\]
Recall that $d_i = \pm 1$ unless $i=n$, in which case $d_n = 2$. Writing $\lambda \in \Lambda_R$ as $\lambda = \sum_{i=0}^n \lambda_i \alpha_i$ with $\lambda_i \in \Z$, we have $\langle \lambda, \alpha_n \rangle = -2 \lambda_{n-1} + 4 \lambda_n$.
Since $r$ is odd, $\langle \lambda, \alpha_n \rangle \in r \Z$ if and only if
$2 \langle \lambda, \alpha_n \rangle \in 2 r \Z$, whence $r \Lambda_W \cap \Lambda_R$ is given by the right-hand side of equation \eqref{eq:freeRealWeights}.
\end{proof}

\begin{Def}
\label{def:freeRealWeights}
Let $\Lambda_{\FR}$ be the additive group which appears on either side of the equality in Lemma \ref{lem:alternativeZ}.
\end{Def}

Write $\mu \in \Lambda_R$ as $\mu = \sum_{i = 0}^n \mu_i \alpha_i \in \Lambda_R$ with $\mu_i \in \Z$. Define $K_{\mu} = \prod_{i = 0}^nK_i^{\mu_i}$. Note that $K_{\alpha_i} = K_i$ and if $V \in \cat$ with $v \in V$ of weight $\lambda \in \Cart^{\vee}$, then $K_{\alpha} v = q^{\langle \lambda, \alpha \rangle}v$.

Define
\begin{equation}
\label{eq:pivotElement}
\pi = 2r\rho_{\p 0} - 2\rho \in \Lambda_R.
\end{equation}

\begin{Lem}
\label{lem:pivotElement}
The element $K_{\pi}$ is a pivot of the Hopf superalgebra $\qhospnres$.
\end{Lem}

\begin{proof} 
Note that $K_{\pi} = K_{2 r \rho_{\p 0}} K_{-2 \rho}$ with $K_{2 r \rho_{\p 0}}$ central (see Lemma \ref{lem:KPowerCentral}). We compute
\[
K_{\pi}E_iK_{\pi}^{-1} = K_{-2\rho}E_iK_{-2\rho}^{-1} = q^{-\langle 2\rho, \alpha_i \rangle}E_i = q^{-\langle \alpha_i, \alpha_i \rangle}E_i = S^2(E_i),
\]
the third equality following from Lemma \ref{lem:rhoPairings}\eqref{ite:evRhoSimp}. Similarly, we find $K_{\pi}F_iK_{\pi}^{-1} = S^2(F_i)$.
\end{proof}

Given $V \in \cat$, let $V^{\vee} \in \cat$ be the $\C$-linear dual of the underlying super vector space of $V$ with $\qhospnres$-module structure given by
\begin{equation*}
(x \cdot f)(v) = (-1)^{\p f \p x}f(S(x)v), \qquad v \in V, \, f \in V^{\vee}, \, x \in \qhospnres.
\end{equation*}
Let $\{v_i\}_i$ be a homogeneous basis of $V$ with dual basis $\{v_i^{\vee}\}_i$. Direct computations show that the maps
\begin{equation*}
\begin{split}
\tev_V (f \otimes v) = f(v), \qquad & \ev_V(v \otimes f) = (-1)^{\p f \p v}f(K_{\pi}v),\\
\tcoev_V(1) = \sum_i v_i \otimes v_i^{\vee} ,\qquad & \coev_V(1) = \sum_i (-1)^{\p v_i}v_i^{\vee} \otimes K_{\pi}^{-1} v_i,
\end{split}
\end{equation*}
define a pivotal structure on $\cat$ with pivotal isomorphism induced by the element $K_{\pi}$ of Lemma \ref{lem:pivotElement}.

\begin{Rem}
\label{rem:differencePivotal}
The element $\pi$ defined in equation \eqref{eq:pivotElement} should be seen as the $\ospn$-analogue of that defined in \cite[\S 3.3]{anghel2021} for $\mathfrak{sl}(m \vert n)$. The differing sign in the definitions for $\ospn$ and $\mathfrak{sl}(m \vert n)$ reflects the differing Hopf structures; see Remark \ref{rem:compareHopfStr}.
\end{Rem}

Let $\overline{U}_q^{H}(\mathfrak{p})$ be the subalgebra of $\qhospnres$ generated by $H_i$, $K_i^{\pm 1}$, $E_i$, $0 \leq i \leq n$, and $F_i$, $1 \leq i \leq n$. We view a weight $\overline{U}_q^H(\ospn_{\p 0})$-module as a $\overline{U}_q^{H}(\mathfrak{p})$-module by placing it in degree $\p 0$ and letting $E_0$ act by zero.

\begin{Def}
\label{def:kacModule}
The \emph{Kac module} of a weight $\overline{U}_q^H(\ospn_{\p 0})$-module $W$ is
\[
K(W) = \qhospnres \otimes_{\overline{U}_q^{H}(\mathfrak{p})} W.
\]
\end{Def}

Since $\vert \Delta_{\p 1}^- \vert = 2n$, the Poincar\'{e}--Birkhoff--Witt Theorem gives $\dim_{\C} K(W) = 2^{2n} \cdot \dim_{\C} W$. When $W = S_0(\lambda)$ is simple of highest weight $\lambda \in \Cart^{\vee}$, we write $K(\lambda)$ for $K(S_0(\lambda))$. The module $S_0(\lambda)$ is a quotient of the Verma module $V_0(\lambda)$ of highest weight $\lambda$. Moreover, $V_0(\lambda)$ is itself simple for generic $\lambda$, in which case $S_0(\lambda) = V_0(\lambda)$; see \cite[Thm. 3.2]{deconcini1990}, \cite[Prop. 34]{geer2013b} for a discussion of genericity.

\begin{Lem}[{\emph{Cf}. \cite[Lem. 3.7]{dejeugt1990}}]
\label{lem:submoduleOfKac}
Any non-zero submodule of $K(\lambda)$ contains the $\overline{U}_q(\ospn_{\p 0})$-module $\Gamma^-_{\p{1}} \cdot S_0(\lambda)$.
\end{Lem}

\begin{proof}
A direct calculation shows that $[H_0, \Gamma^-_{\p 1}]= 2n \Gamma^-_{\p 1}$. Together with Lemma \ref{lem:gammaCentral}, this implies that $\Gamma^-_{\p 1} \cdot S_0(\lambda)$ is indeed a $\overline{U}_q(\ospn_{\p 0})$-submodule of $K(\lambda)$.

Let $w \in K(\lambda)$ be non-zero. To prove the lemma, it suffices to prove that the $\qhospnres$-submodule generated by $w$ contains $\Gamma^-_{\p 1} \cdot S_0(\lambda)$. By the Poincar\'{e}--Birkhoff--Witt Theorem, we can write $w$ as a sum of non-zero terms
\[
w = \sum_{\vec{\alpha}} F_{\vec{\alpha}} \cdot v_{\vec{\alpha}}
\]
indexed by strictly decreasing (with respect to $<$) subsets $\vec{\alpha} \subset \Delta_{\p 1}^+$. Here $F_{\vec{\alpha}}$ is the corresponding ordered product of odd root vectors and $v_{\vec{\alpha}} \in S_0(\lambda)$. Let $\vec{\alpha}_0$ be a minimal (by length) subset which appears in the sum and $\vec{\alpha}_0^c =  \Delta^+_{\p 1} \setminus \vec{\alpha}_0$ its complement. By \cite[Lem. 5(1)]{zhang1993}, the vector $F_{\vec{\alpha}_0^c} \cdot w$ is proportional to $\Gamma^-_{\p 1} \cdot v_{\vec{\alpha}_{0}}$. The submodule generated by $w$ therefore contains an element of $\Gamma^-_{\p 1} \cdot S_0(\lambda)$, and hence the entirety of $\Gamma^-_{\p 1} \cdot S_0(\lambda)$ since $S_0(\lambda)$ is a simple $\overline{U}_q(\ospn_{\p 0})$-module.
\end{proof}

The Kac module $K(\lambda)$ is called \emph{typical} if it is simple. We use the same terminology for the weight $\lambda \in \Cart^{\vee}$.

\begin{Prop}[{\emph{Cf.} \cite[Prop. 2]{zhang1993}}]
\label{prop:kacTypical}
The Kac module $K(\lambda)$ is typical if and only if
\[
\prod_{\alpha \in \Delta_{\p 1}^+} \left\{ \langle \lambda+\rho,\alpha \rangle \right\}_q \neq 0.
\]
\end{Prop}

\begin{proof}
Let $v_0 \in S_0(\lambda)$ be a highest-weight vector and $w \in K(\lambda)$ non-zero. By Lemma \ref{lem:submoduleOfKac}, the submodule generated by $w$ contains the non-zero vector $\Gamma^-_{\p 1} \cdot v_0$ and hence also $\Gamma^+_{\p 1} \Gamma^-_{\p 1} \cdot v_0$. By Lemma \ref{lem:gammaCentral} and Schur's Lemma, $\Gamma^+_{\p 1} \Gamma^-_{\p 1} \cdot v_0$ is a scalar multiple of $v_0$. More precisely, \cite[Lem. 9]{zhang1993} shows that $\Gamma^+_{\p 1} \Gamma^-_{\p 1} \cdot v_0 = c \chi(\lambda) v_0$, where $c \in \C^{\times}$ and $\chi(\lambda) = \prod_{\alpha \in \Delta_{\p 1}^+} \left\{ \langle \lambda+\rho,\alpha \rangle \right\}_q$. If $\chi(\lambda) = 0$, then $\Gamma^-_{\p 1} \cdot v_0$ generates a strict submodule of $K(\lambda)$ (since it does not contain $v_0$), whence $K(\lambda)$ is not simple. If $\chi(\lambda) \neq 0$, then $\Gamma^-_{\p 1} \cdot v_0$ generates a submodule which contains $v_0$ and hence is the entirety of $K(\lambda)$, whence $K(\lambda)$ is simple.
\end{proof}


\begin{Lem}
\label{lem:dualKac}
If $\lambda \in \Cart^{\vee}$ is typical and $S_0(\lambda) = V_0(\lambda)$ is a simple Verma $\overline{U}^H_q(\ospn_{\p 0})$-module, then there is an isomorphism $K(\lambda)^{\vee} \simeq K(-\lambda + \pi)$.
\end{Lem}

\begin{proof}
First, note that $V_0(\lambda)^{\vee}$ is simple. The Poincar\'{e}--Birkhoff--Witt Theorem implies that the lowest weight of $K(\lambda)$ is
\[
\lambda - 2(r-1) \rho_{\p 0} - 2 \rho_{\p 1}
=
\lambda - \pi.
\]
The highest weight of $K(\lambda)^{\vee}$ is therefore $- \lambda + \pi$. Since $-\lambda + \pi + \rho = - \lambda - \rho + 2r \rho_{\p 0}$ and $\langle 2 r \rho_{\p 0}, \alpha \rangle \in 2 r \Z$ for all $\alpha \in \Delta_{\p 1}^+$ (see Lemma \ref{lem:rhoPairings}\eqref{ite:oddRhoOdd}), we find
\[
\left\{ \langle - \lambda + \pi + \rho, \alpha \rangle \right\}_q
=
- \left\{\langle \lambda + \rho, \alpha \rangle \right\}_q.
\]
Hence, typicality of $\lambda$ implies typicality of $-\lambda + \pi$. Since the parity of the highest-weight vector of $K(\lambda)$ is $\p 0$, the parity of its lowest-weight vector is the parity of $\vert \Delta_{\p 1}^+ \vert = 2n$, which is again $\p 0$. It follows that the map which assigns a highest-weight vector of $K(\lambda)^{\vee}$ to one of $K(-\lambda + \pi)$ extends to a $\qhospnres$-linear isomorphism.
\end{proof}

\subsection{Generic semisimplicity}

Let $\Gr$ be the abelian group $\Cart^{\vee} \slash \Lambda_R$. Write $[\lambda] \in \Gr$ for the class of $\lambda \in \Cart^{\vee}$. Let $\cat_{[\lambda]} \subset \cat$ be the full subcategory of modules whose weights lie in the class $[\lambda]$ and $\cat = \bigoplus_{[\lambda] \in \Gr} \cat_{[\lambda]}$ the resulting $\Gr$-grading.

Let $\overline{U}^{< 0}_q(\ospn)$ be the subalgebra of $\qhospnres$ generated by $F_i$, $0 \leq i \leq n$. The Poincar\'{e}--Birkhoff--Witt Theorem gives $\dim_{\C} \overline{U}^{< 0}_q(\ospn)
= r^{n^2}2^{2n}$. Note that this dimension is an upper bound of the dimension of a highest-weight $\qhospnres$-module.

\begin{Def}
Let $\SSS \subset \Gr$ be the set of those $[\lambda] \in \Gr$ such that $\cat_{[\lambda]}$ contains a simple module of dimension strictly less than $r^{n^2}2^{2n}$.
\end{Def}

\begin{Lem}
\label{lem:smallSymmetricSubset}
 The set $\SSS$ is small and symmetric in $\Gr$.
\end{Lem}

\begin{proof}
Symmetry follows from the observation that duality $(-)^{\vee}: \cat_{[\lambda]} \rightarrow \cat_{-[\lambda]}$ preserves the dimension and simplicity of a module. Consider the set
\[
\mathsf{Y} = \{ \lambda \in \Cart^{\vee} \mid V_0(\lambda) \mbox{ is a simple } \overline{U}_q^H(\ospn_{\p 0})\mbox{-module and } \lambda \mbox{ is typical}\}. 
\]
The characterization of simple Verma $\overline{U}_q^H(\ospn_{\p 0})$-modules \cite[Prop. 17]{geer2013b} together with Proposition \ref{prop:kacTypical} shows that $\mathsf{Y} \subset \Cart^{\vee}$ is open and dense. Any simple module $V \in \cat_{[\lambda]}$, $\lambda \in \mathsf{Y}$, has dimension $r^{n^2}2^{2n}$. Hence, $\SSS$ is contained in the complement of the image of $\mathsf{Y}$ in $\Gr$ and is therefore small. 
\end{proof}

We recall the following result.

\begin{Prop}[{\cite[Thm. 3.2.2]{etingof2011}}]
\label{prop:density}
Let $A$ be a $\C$-algebra and $\{\rho_j: A \rightarrow \End_{\C}(V_j)\}_{j \in J}$ a finite set of pairwise non-isomorphic finite-dimensional simple $A$-modules. Then the map $\bigoplus_{j \in J} \rho_j: A \rightarrow \bigoplus_{j \in J} \End_{\C}(V_j)$ is surjective. 
\end{Prop}

\begin{Thm}
\label{thm:genericSemiSimple}
The $\Gr$-graded category $\cat$ is generically semisimple with critical set $\SSS$.
\end{Thm}

\begin{proof}
The proof is a modification of that of \cite[Thm. 3.23]{anghel2021}. Fix $\lambda = \sum_{i=0}^n \lambda_i \alpha_i \in \Cart^{\vee}$ such that $\left[ \lambda \right] \in \Gr \setminus \SSS$. We prove that $\cat_{\left[ \lambda \right]}$ is semisimple.

We begin with a preliminary construction. View $\qhospnres$ as a module over the group algebra $\C[\Lambda_R]$ by letting $\alpha_i \in \Delta^+$ act by $K_i$. Let $\overline{U}^{\Lambda_W}_q(\ospn)$ be the subalgebra of $\C[\Lambda_W] \otimes_{\C[\Lambda_R]} \qhospnres$ generated by the fundamental dominant weights $K^{\pm 1}_{w_i}$ and $E_i$, $F_i$, $0 \leq i \leq n$. Consider the two-sided ideal $I^{[\lambda]}= \langle K_{w_i}^{r} - q_i^{r\lambda_i} \mid 0 \leq i \leq n \rangle$ and the finite-dimensional quotient
\[
\overline{U}^{[ \lambda ]}_q(\ospn) = \overline{U}^{\Lambda_W}_q(\ospn) \slash I^{[\lambda]}.
\]

Any object of $\cat_{[ \lambda ]}$ is naturally a module over $\overline{U}^{[ \lambda ]}_q(\ospn)$. Indeed, let $V \in \cat_{[\lambda]}$ with $v \in V$ of weight $\mu$, so that $K^r_{w_i} v = q^{r\langle w_i, \mu \rangle} v$. Since $\mu - \lambda \in \Lambda_R$, we have $\langle w_i, \mu \rangle \equiv \langle w_i, \lambda \rangle \mod \Z$ so that $q^{r \langle w_i, \mu \rangle} = q^{r \langle w_i, \lambda \rangle}$. Since $\langle w_i, \lambda \rangle = d_i \lambda_i$, we have $q^{r \langle w_i, \mu \rangle} = q^{r d_i \lambda_i}$. We conclude that $K_{w_i}^r - q_i^{r \lambda_i}$ acts by zero on $V$.

Continuing, let $V,V^{\prime} \in \cat_{[ \lambda ]}$ be simple modules that are isomorphic as $\overline{U}^{[ \lambda ]}_q(\ospn)$-modules. We claim that $V \simeq V^{\prime} \otimes \sigma(z, \p 0)$ for some $z \in \Lambda_{\FR}$. To see this, let $v \in V$ and $v^{\prime} \in V^{\prime}$ be highest-weight vectors of weights $\lambda$ and $\lambda^{\prime}$, respectively. Without loss of generality, we may assume that the $\overline{U}^{[ \lambda ]}_q(\ospn)$-linear isomorphism $V \rightarrow V^{\prime}$ sends $v$ to $v^{\prime}$. Agreement of the actions of $K_{w_i}$ under this isomorphism requires $q^{\langle w_i, \lambda \rangle} = q^{\langle w_i, \lambda^{\prime} \rangle}$,
equivalently
\[
\lambda - \lambda^{\prime} \in \{\mu \in \Lambda_R \mid \langle w_i, \mu \rangle \in r \Z \mbox{ for all } \; 0 \leq i \leq n\}.
\]
Writing $\mu = \sum_{j=0}^n \mu_j \alpha_j \in \Lambda_R$ with $\mu_j \in \Z$, we have $\langle w_i, \mu \rangle = d_i \mu_i$. The conditions $\langle w_i, \mu \rangle \in r \Z$, $0 \leq i \leq n$, therefore imply $\langle \alpha_i, \mu \rangle = \sum_{j=0}^n \mu_j d_i a_{ij} \in r \Z$. From this and the proof of Lemma \ref{lem:alternativeZ}, we see that $\lambda - \lambda^{\prime} \in \Lambda_{\FR}$, whence $V \simeq V^{\prime} \otimes \sigma(\lambda - \lambda^{\prime}, \p 0)$.

Suppose now that $\mu \in \Cart^{\vee}$ is the highest weight of a simple $\overline{U}^{[\lambda ]}_q(\ospn)$-module $V$, so that $K_{w_i} v_0 = q^{\langle w_i, \mu \rangle} v_0$ for a highest-weight vector $v_0 \in V$. Note that $V$ depends only on the residue class of $\mu$ modulo $r \Lambda_R$. The module $V$ lifts to $\cat_{[\lambda]}$ if and only if $q_i^{r\mu_i} = q_i^{r\lambda_i}$ for all $0 \leq i \leq n$. There are exactly $r^{n+1}$ inequivalent choices of $\mu$ which satisfy these equations; denote by $\{V_j\}_{j \in J}$ the corresponding $\overline{U}^{[\lambda ]}_q(\ospn)$-modules. Since $[\lambda] \in \Gr \setminus \SSS$, each module $V_j$ lifts to a Kac module of dimension $D = r^{n^2}2^{2n}$. Proposition \ref{prop:density} ensures that the canonical map $\overline{U}^{[ \lambda ]}_q(\ospn) \rightarrow \bigoplus_{j \in J}\End_{\C}(V_j)$ is a surjection. A dimension count shows that the domain and codomain of this map each have dimension $r^{n+1} D^2 = r^{n + 1} r^{2n^2}2^{4n}$, whence the map is an isomorphism. By Wedderburn's Theorem, $\overline{U}^{[ \lambda ]}_q(\ospn)$ is a semisimple algebra. 

We can now complete the proof. Let $V \in \cat_{\left[ \lambda \right]}$ be non-zero. From the above discussion, $\dim_{\C} V = mD$, where $m \in \Z_{>0}$ is the dimension of the super vector space of highest-weight vectors of $V$ for the action of $\overline{U}^{[ \lambda ]}_q(\ospn)$. Fixing a homogeneous basis $v_1, \dots, v_m$ of this space, we obtain a $\overline{U}^{[ \lambda ]}_q(\ospn)$-linear isomorphism
\[
\bigoplus_{i=1}^m \overline{U}^{[ \lambda ]}_q(\ospn) \cdot v_i \rightarrow V.
\]
Since $\qhospnres v_i = \overline{U}^{[ \lambda ]}_q(\ospn) v_i$, this lifts to an isomorphism
\[
\bigoplus_{i=1}^m \qhospnres \cdot v_i \rightarrow V
\]
in $\cat$. Finally, $\qhospnres \cdot v_i$ is a simple $\qhospnres$-module since $[\lambda] \in \Gr \setminus \SSS$.
\end{proof}

\begin{Prop}
\label{prop:kacProjective}
If $\lambda \in \Cart^{\vee}$ is typical and $V_0(\lambda) = S_0(\lambda)$ is a simple Verma module, then $K(\lambda) \in \cat$ is projective and injective.
\end{Prop}

\begin{proof}
Under the stated assumptions, $[\lambda] \in \Gr \setminus \SSS$. Hence, $K(\lambda)$ is an object of the semisimple category $\cat_{[\lambda]}$ (see Theorem \ref{thm:genericSemiSimple}) and so is projective and injective.
\end{proof}

\subsection{Ribbon structure}
\label{sec:ribCat}

We follow the method of \cite[\S 3.2]{anghel2021} to construct a braiding on the category $\cat$. Our starting point is Yamane's universal $R$-matrix for the $h$-adic quantum group $U_h(\ospn)$ \cite[Thm. 10.6.1]{yamane1994}.
Keeping the notation of Section \ref{sec:ospn}, let $(d_{ij})$ be the inverse of the matrix $(a_{ij} \slash d_j)$ and define 
\[
\mathcal{H}^h = \q^{-\sum_{i,j=0}^n d_{ij}H_i \otimes H_j}
\in
U_h(\ospn) \widehat{\otimes} U_h(\ospn),
\]
where $\q = e^{\frac{h}{2}}$ for a formal parameter $h$. Define the $\q$-exponential $\exp_{\q} (x)= \sum_{k = 0}^{\infty} \frac{x^k}{(k)_{\q}!}$, where $(k)_{\q} = \frac{1-\q^k}{1-\q}$. Then the universal $R$-matrix is $\mathcal{R}^h = \check{\mathcal{R}}^h\mathcal{H}^h$ with quasi-$R$-matrix
\[
\check{\mathcal{R}}^h = \prod_{\alpha \in \Delta^+} \exp_{\q_{\alpha}}\left( c_{\alpha}(\q) (\q-\q^{-1}) E_{\alpha} \otimes F_{\alpha} \right)
\in
U_h(\ospn) \widehat{\otimes} U_h(\ospn).
\]
Here $\q_{\alpha} = \epsilon(\alpha)\q^{-\langle \alpha,\alpha \rangle}$ and $c_{\alpha}(\q) = (-1)^i\q^k$ for some $\alpha$-dependent integers $i,k \in \Z$; see \cite[Lem. 10.3.1 and 10.6.1]{yamane1994}.

The Poincar\'{e}--Birkhoff--Witt topological basis of $U_h(\ospn)$ is
\[
\mathcal{B}
=
\left\{\
\prod_{i=0}^n H_i^{z_i} \prod_{\alpha \in \Delta^+} E_{\alpha}^{x_{\alpha}} \prod_{\alpha \in \Delta^+} F_{\alpha}^{y_{\alpha}}
\mid z_i, x_{\alpha}, y_{\alpha} \in \Z_{\geq 0}, \, x_{\alpha}, y_{\alpha} \leq 1 \mbox{ if } \alpha \in \Delta^+_{\p 1} \right\}.
\]
Let $\mathcal{B}^< \subset \mathcal{B}$ be the subset of monomials with $x_{\alpha}, y_{\alpha} \leq r-1$ for all $\alpha \in \Delta^+_{\p 0}$. There is a vector space decomposition $U_h(\ospn) = U_h(\ospn)^< \oplus I$, where $U_h(\ospn)^<$ and $I$ are the closures in $U_h(\ospn)$ of $\Span_{\C}(\mathcal{B}^<)$ and $\Span_{\C}(\mathcal{B} \slash \mathcal{B}^<)$, respectively. Let $\rho: U_h(\ospn) \rightarrow U_h(\ospn)^<$ be the projection. Define the \emph{truncated $h$-adic $R$-matrix} $\mathcal{R}^{<h} = (\rho \otimes \rho)\mathcal{R}^h$.

\begin{Lem}
\label{lem:truncRMatrixIden}
The truncated $h$-adic $R$-matrix satisfies the following identities:
\begin{enumerate}
\item $(\rho \otimes \rho \otimes \rho)(\Delta \otimes \id)\mathcal{R}^{< h} = (\rho \otimes \rho \otimes \rho)(\mathcal{R}_{13}^{<h}\mathcal{R}_{23}^{<h})$.

\item $(\rho \otimes \rho \otimes \rho)(\id \otimes \Delta)\mathcal{R}^{< h} = (\rho \otimes \rho \otimes \rho)(\mathcal{R}_{13}^{<h}\mathcal{R}_{12}^{<h})$.

\item $(\rho \otimes \rho) \mathcal{R}^{< h} (\Delta^{\op}(x)) = (\rho \otimes \rho)(\Delta(x))\mathcal{R}^{< h}$. 
\end{enumerate}
\end{Lem}

\begin{proof}
These identities follow from the corresponding identities for the universal $R$-matrix $\mathcal{R}^h$, as in the proof of \cite[Lem. 3.9]{anghel2021}.
\end{proof}

Returning to the case in which $q$ is a root of unity, consider the operator $\mathcal{H}^q = q^{-\sum_{i,j = 0}^n d_{ij} H_i \otimes H_j}$. Concretely, if $V, W \in \cat$ with $v \in V$ and $w \in W$ of weight $\lambda$ and $\mu$, respectively, then
\begin{equation}
\label{eq:HAction}
\mathcal{H}^q(v \otimes w)
=
q^{-\sum_{i,j=0}^n d_{ij} \lambda(H_i) \mu(H_j)} v \otimes w
=
q^{- \langle \lambda, \mu \rangle}v \otimes w.
\end{equation}

\begin{Lem}
\label{lem:HIden}
The following equalities hold as operators on $\cat$.
\begin{enumerate}
\item \label{ite:commuteHK} $\mathcal{H}^q(x \otimes y) = q^{- \langle \alpha,\beta \rangle}(xK^{-1}_{\beta} \otimes yK^{-1}_{\alpha})\mathcal{H}^q$ for $x,y  \in \qhospnres$ of weight $\alpha,\beta \in \Cart^{\vee}$, respectively. 

\item \label{ite:firstBraidEqn} $(\Delta \otimes \id)\mathcal{H}^q = \mathcal{H}^q_{13}\mathcal{H}^q_{23}$. 

\item \label{ite:secondBraidEqn} $(\id \otimes \Delta)\mathcal{H}^q = \mathcal{H}^q_{13} \mathcal{H}^q_{12}$.
\end{enumerate}
\end{Lem}
\begin{proof}
Let $v \in V$ and $w \in W$ of weight $\lambda$ and $\mu$, respectively. Equation \eqref{eq:HAction} gives
$\mathcal{H}^q(x \otimes y)(v \otimes w) = q^{- \langle \alpha + \lambda ,\beta + \mu \rangle} xv \otimes yw$. On the other hand, we compute
\begin{eqnarray*}
q^{-\langle \alpha,\beta\rangle}(xK^{-1}_{\beta} \otimes y K^{-1}_{\alpha})\mathcal{H}^q(v \otimes w)
&=& q^{- \langle \alpha,\beta\rangle - \langle \lambda , \mu \rangle}(x K^{-1}_{\beta} \otimes y K^{-1}_{\alpha})(v \otimes w)\\
&=& q^{- \langle \alpha + \lambda ,\beta + \mu \rangle} xv \otimes yw.
\end{eqnarray*}
    
To prove the second identity, we compute
\begin{eqnarray*}
(\Delta \otimes \id)\mathcal{H}^q
&=&
\sum_{k=0}^{\infty} \frac{1}{k!} \left(- \frac{2 \pi i}{r}\right)^k \left( \sum_{i,j=0}^n d_{ij} (H_i \otimes 1 + 1 \otimes H_i) \otimes H_j \right)^k. \\
\end{eqnarray*}
Applying this operator to $u \otimes v \otimes w$, where $u$, $v$ and $w$ have weight $\lambda$, $\mu$ and $\nu$, respectively, gives 
\[
\sum_{k=0}^{\infty} \frac{1}{k!} \left(- \frac{2 \pi i}{r}\right)^k \left( \sum_{i,j=0}^n d_{ij} (\lambda(H_i) + \mu(H_i)) \nu(H_i) \right)^k u \otimes v \otimes w
=
q^{- \langle \lambda +\mu, \nu \rangle} u \otimes v \otimes w.
\]
On the other hand, we have
\[
\mathcal{H}^q_{13}\mathcal{H}^q_{23}(u \otimes v \otimes w)
=
q^{- \langle \lambda, \nu \rangle} \mathcal{H}^q_{13} (u \otimes v \otimes w)
=
q^{- \langle \lambda, \nu \rangle -\langle \mu,\nu \rangle} u \otimes v \otimes w.
\]
This proves the second identity. The third identity is proved in the same way.
\end{proof}

The truncated $q$-exponential is $\exp_{q}^<(x)= \sum_{k = 0}^{r - 1}\frac{x^k}{(k)_q!}$. Define $\mathcal{R}^q = \check{\mathcal{R}}^q \circ \mathcal{H}^q$, where
\[
\check{\mathcal{R}}^q = \prod_{\alpha \in \Delta^+} \exp^<_{q_{\alpha}}\left( c_{\alpha}(q) (q-q^{-1}) E_{\alpha} \otimes F_{\alpha} \right) \in \qhospnres \otimes \qhospnres
\]
is seen as an operator by left multiplication.

\begin{Lem}
\label{lem:modRMatrixIden}
The following equalities hold as operators on $\cat$.
\begin{enumerate}
\item $(\Delta \otimes \id) \mathcal{R}^q = \mathcal{R}^q_{13}\mathcal{R}^q_{23}$,
\item $(\id \otimes \Delta) \mathcal{R}^q = \mathcal{R}^q_{13}\mathcal{R}^q_{12}$,
 \item $\mathcal{R}^q\Delta(x) = \Delta^{\op}(x)\mathcal{R}^q$ for any $x \in \qhospnres$.
\end{enumerate}
\end{Lem}

\begin{proof}
After noting that $\check{\mathcal{R}}^q$ is the specialization $\q \mapsto q$ of the truncated $h$-adic quasi-$R$-matrix $\check{\mathcal{R}}^{< h} = (\rho \otimes \rho)\check{\mathcal{R}}^h$, this follows from Lemmas \ref{lem:truncRMatrixIden} and \ref{lem:HIden}, as in the proof of \cite[Lem. 3.15]{anghel2021}.
\end{proof}

Lemma \ref{lem:modRMatrixIden} immediately leads to the following result.

\begin{Cor}
\label{cor:braidedCat}
A braiding on the category $\cat$ is defined by the maps
\[
c_{V,W}: V \otimes W \rightarrow W \otimes V,
\qquad
V,W \in \cat
\]
where $c_{V,W} = \tau \circ \mathcal{R}^q$ and $\tau(v \otimes w) = (-1)^{\p{v} \p{w}} w \otimes v$ is the swap map with Koszul signs. 
\end{Cor}

Now that $\cat$ is endowed with a pivotal structure and a braiding, we can define a candidate ribbon structure $\theta$ on $\cat$ by
\[
\theta_V = (\id_V \otimes \ev_V) \circ (c_{V,V} \otimes \id_{V^{\vee}}) \circ (\id_V \otimes \tcoev_V),
\qquad
V \in \cat.
\]

\begin{Prop}
\label{prop:ribCat}
Let $V \in \cat$ be a highest-weight module with highest weight $\lambda$. Then $\theta_V = q^{-\langle \lambda - \pi, \lambda \rangle} \id_V$. Moreover, if $[\lambda] \in \Gr \setminus \SSS$, then $\theta_{K(\lambda)^{\vee}} = \theta^{\vee}_{K(\lambda)}$.
\end{Prop}

\begin{proof} 
Let $\{v_i\}_i$ be a homogeneous weight basis of $V$ with $v_i$ of weight $\lambda_i$. Fix $\lambda_0 = \lambda$, so that $v_0$ is highest-weight. We compute
\begin{eqnarray*}
\theta_V v_0
&=&
(\id_V \otimes \ev_V) \circ (\tau \circ \mathcal{R}^q  \otimes \id_{V^{\vee}})\sum_i v_0 \otimes  v_i \otimes v_i^{\vee} \\
&=&
(\id_V \otimes \ev_V)\sum_i (-1)^{\p{v}_0 \p{v}_i} q^{-\langle \lambda, \lambda_i \rangle} v_i \otimes v_0 \otimes v_i^{\vee}\\
&=&
\sum_i (-1)^{\p{v}_0 \p{v}_i + \overline{v_0} \p{v}_i} q^{\langle -\lambda, \lambda_i \rangle + \langle \pi,\lambda \rangle} \delta_{i0}v_i
=
q^{- \langle \lambda - \pi,\lambda \rangle} v_0 .
\end{eqnarray*}
The second equality follows from the fact that $v_0$ is highest-weight and the explicit form of $\check{\mathcal{R}}^q$. The first statement now follows from the isomorphism $\End_{\cat}(V) \simeq \C \cdot \id_V$. The second statement follows from the first statement and Lemma \ref{lem:dualKac}.
\end{proof}

In view of Theorem \ref{thm:genericSemiSimple}, Corollary \ref{cor:braidedCat} and Proposition \ref{prop:ribCat}, we may apply \cite[Thm. 9]{geer2018} to deduce that $\cat$ is a ribbon category.

\subsection{Modified traces}
\label{sec:modTrace}

\begin{Prop}
\label{prop:unimodularity}
The category $\cat$ is unimodular. 
\end{Prop}

\begin{proof}
We need to prove that the injective hull of the monoidal unit $\C \in \cat$ is self-dual.

Let $V = K(\lambda) \in \cat$  be a generic simple object with highest- and lowest-weight vectors $v_+$ and $v_-$ of weights $\lambda$ and $\lambda - \pi$, respectively; see Lemma \ref{lem:dualKac}. Since $V$ is projective (Proposition \ref{prop:kacProjective}), there is a decomposition into projective indecomposables,
\[
V \otimes V^{\vee} \simeq \bigoplus_i P_i.
\]
Adjunction and Schur's Lemma give
\[
\Hom_{\cat}(\C, \bigoplus_iP_i)
\simeq
\End_{\cat}(V)
\simeq
\C.
\]
We deduce that the injective hull of $\C$ appears exactly once in $\bigoplus_i P_i$, call it $P_0$, and that $P_0$ is the unique summand which contains a non-zero invariant vector. Since the weight spaces of weights $\pi$ and $-\pi$ in $V \otimes V^{\vee}$ are each one-dimensional, $v_+ \otimes v_-^{\vee} \in P_i$ and $v_- \otimes v_+^{\vee} \in P_j$ for some indices $i$ and $j$. Since $(V \otimes V^{\vee})^{\vee} \simeq V \otimes V^{\vee}$, we have $P_i^{\vee} \simeq P_j$.

Keeping the notation of Section \ref{sec:quantGroup}, consider the element $\Gamma^{\pm} = \Gamma^{\pm}_{\p 1} \Gamma^{\pm}_{\p 0}$ of weight $\pm \pi$. We claim that $\Gamma^-(v_+ \otimes v_-^{\vee})$ and $\Gamma^+(v_- \otimes v_+^{\vee})$ are non-zero invariant vectors. From this it will follow that $i=j=0$ and $P_0$ is self-dual. We prove the claim for $\Gamma^+(v_- \otimes v_+^{\vee})$, omitting the similar proof for $\Gamma^-(v_+ \otimes v_-^{\vee})$.

As $\Gamma^+(v_- \otimes v_+^{\vee})$ has weight $0$, it is invariant under $H_i$. Invariance under $E_0$ follows from the vanishing of $E_0 \Gamma^+$, a consequence of \cite[Lem. 5(1)]{zhang1993}. For $1 \leq i \leq n$, Lemma \ref{lem:gammaCentral} gives $E_i \Gamma^+ = \Gamma^+_{\p 1} E_i \Gamma^+_{\p 0}$. Up to a non-zero scalar, the element $\Gamma^+_{\p 0}$ is independent of the convex ordering on $\Delta^+_{\p 0}$ used in its definition; see \cite[Lem. 37]{geer2013b}. In particular, choosing an ordering in which $\alpha_i$ is first shows that $E_i \Gamma^+_{\p 0} =0$. Continuing, we compute
\begin{equation}
\label{eq:FZeroApplied}
F_0 \Gamma^+ (v_- \otimes v_+^{\vee})
=
-E_0 F_0 \tilde{\Gamma}^+_{\p 1} \Gamma^+_{\p 0} (v_- \otimes v_+^{\vee}) + \frac{K_0 - K_0^{-1}}{q_0 - q_0^{-1}} \tilde{\Gamma}^+_{\p 1} \Gamma^+_{\p 0} (v_- \otimes v_+^{\vee}),
\end{equation}
where $\tilde{\Gamma}^+_{\p 1}$ is defined similarly to $\Gamma^+_{\p 1}$ but with the root $\alpha_0$ omitted from the product.
Since $\tilde{\Gamma}^+_{\p 1} \Gamma^+_{\p 0} (v_- \otimes v_+^{\vee})$ has weight $-\alpha_0$ 
and $\langle \alpha_0, \alpha_0 \rangle =0$, the second term in the right-hand side of equation \eqref{eq:FZeroApplied} vanishes. Note that weights in the submodule generated by $v_- \otimes v^{\vee}_+$ are of the form $-\pi + \sum_{\alpha \in \Delta^+} c_{\alpha} \alpha$, where $0 \leq c_{\alpha} \leq m_{\p \alpha}$. Since $- 2 \alpha_0$, the weight of $F_0 \tilde{\Gamma}^+_{\p 1} \Gamma^+_{\p 0} (v_- \otimes v_+^{\vee})$, is not of this form, this vector vanishes, proving invariance under $F_0$. For $1 \leq i \leq n$, Lemma \ref{lem:gammaCentral} gives
\[
F_i \Gamma^+ (v_- \otimes v_+^{\vee})
=
\Gamma^+_{\p 1} F_i \Gamma^+_{\p 0} (v_- \otimes v_+^{\vee}).
\]
As above, by ordering $\Delta^+_{\p 0}$ so that $\alpha_i$ is first, we may write $\Gamma^+_{\p 0} = E^{r-1}_i \tilde{\Gamma}^+_{\p 0}$, where now $E^{r-1}_i$ is omitted from the product $\tilde{\Gamma}^+_{\p 0}$. Standard calculations (see \cite[\S 1.3]{jantzen1996}) give
\[
F_i \Gamma^+_{\p 0}
=
E_i^{r-1} F_i \tilde{\Gamma}^+_{\p 0} - [r-1]_{q_i} E^{r-2}_i [K_i; r-2] \tilde{\Gamma}^+_{\p 0}. 
\]
Since $\tilde{\Gamma}^+_{\p 0} (v_- \otimes v_+^{\vee})$ has weight $-(r-1) \alpha_i-2 \rho_{\p 1}$, the action of the second summand on $v_- \otimes v^{\vee}_+$ is proportional to $[\langle \alpha_i, -(r-1) \alpha_i - 2\rho_{\p 1} \rangle + d_i(r-2)]_q$ which, using Lemma \ref{lem:rhoPairings}\eqref{ite:oddRhoSimp}, equals $[-r]_{q_i}=0$. Finally, $E_i^{r-1} F_i \tilde{\Gamma}^+_{\p 0} (v_- \otimes v_+^{\vee})$ vanishes by a weight argument similar to that used to treat the first term on the right-hand side of equation \eqref{eq:FZeroApplied}.

It remains to verify that $\Gamma^+(v_- \otimes v_+^{\vee})$ is non-zero. From the form of the coproduct of positive root vectors (see \cite[Lem. 10.1.1]{yamane1994}), we find
\[
\Gamma^+ (v_- \otimes v_+^{\vee})
=
\Gamma^+ v_- \otimes v_+^{\vee} + \cdots,
\]
where the omitted terms are of the form $a \otimes b$, where $b \in V^{\vee}$ has weight strictly greater than $-\lambda + \pi$. By simplicity of $V_0(\lambda)$, the vector $\Gamma^+_{\p 0} v^{\vee}_+$ is a non-zero scalar multiple of $v^{\vee}_-$. Finally, $\Gamma^+_{\p 1} v^{\vee}_-$ is non-zero by the typicality of $\lambda$; see the proof of Proposition \ref{prop:kacTypical}.
\end{proof}

It follows from Proposition \ref{prop:unimodularity} and \cite[Cor. 6.5]{geer2022} that $\cat$ admits a nondegenerate modified trace $\mt$ on its projective ideal. Moreover, $\mt$ is unique up to a non-zero scalar. Before normalizing $\mt$ in Corollary \ref{cor:mtraceNormalized} below, we require some preparation.

Let $\Z [ \Cart^{\vee} ]$ be the group ring of $\Cart^{\vee}$. Write $e^{\lambda} \in \Z [ \Cart^{\vee} ]$ for the basis element corresponding to $\lambda \in \Cart^{\vee}$. Each $\alpha \in \Cart^{\vee}$ defines a ring homomorphism
\[
\phi_{\alpha}: \Z[\Cart^{\vee}] \rightarrow \C,
\qquad
e^{\lambda} \mapsto q^{\langle \alpha , \lambda \rangle}.
\]
The \emph{supercharacter} of $V \in \cat$ is
\[
\sch(V)
=
\sum_{\lambda \in \Cart^{\vee}} \left( \sdim_{\C} V[ \lambda] \right) e^{\lambda} \in \Z[\Cart^{\vee}],
\]
where $\sdim_{\C} V[ \lambda]$ is the super dimension of the $\lambda$-weight space of $V$.

\begin{Ex}
\label{ex:superChar}
Let $\lambda \in \Cart^{\vee}$ be generic. As a super $\Cart$-module, we have $K(\lambda) \simeq \overline{U}_q^{<0}(\ospn) \otimes \C_{\lambda}$, where $\C_{\lambda}$ is in degree $(\lambda, \p 0) \in \Cart^{\vee} \times \Ztwo$. The Poincar\'{e}--Birkhoff--Witt Theorem gives
\[
\sch (K(\lambda))
=
e^{\lambda} \prod_{\alpha \in \Delta_{\p 0}^+}(1 +  e^{- \alpha} + \dots + e^{-(r - 1)\alpha}) \prod_{\alpha \in \Delta_{\p 1}^+}(1 - e^{- \alpha}). \qedhere
\]
\end{Ex}

Given $V, W \in \cat$, let $S^{\prime}(V,W) \in \End_{\cat}(W)$ be the composition
\begin{equation}
\label{eqn:partialtrace}
(\id_{W} \otimes \ev_V) \circ (c_{V,W} \otimes \id_{V^{\vee}}) \circ (c_{W, V} \otimes \id_{V^{\vee}}) \circ (\id_{W} \otimes \tcoev_V).
\end{equation}

\begin{Lem}
\label{lem:longHopf}
Let $V,W \in \cat$ be simple with $W$ of highest weight $\lambda \in \Cart^{\vee}$. Then we have
\[
\langle S^{\prime}(V,W) \rangle = \phi_{-2\lambda + \pi}(\sch(V)).
\]
If, moreover, $\mu \in \Cart^{\vee}$ is generic, then
\begin{equation*}
\langle S^{\prime}(K(\mu),W) \rangle
=
q^{- 2 \langle \lambda - \frac{\pi}{2},\mu - \frac{\pi}{2} \rangle} \frac{\prod_{\alpha \in \Delta_{\p 0}^+} \{{r\langle\lambda - \frac{\pi}{2},\alpha \rangle}\}_q}{\prod_{\alpha \in \Delta^+} \{\langle \lambda - \frac{\pi}{2}, \alpha \rangle\}_q^{\epsilon(\alpha)}}.
\end{equation*}
\end{Lem}

\begin{proof}
The proof follows that of \cite[Prop. 2.2]{geer2010} with modifications reflecting the differing pivotal structures; see Remark \ref{rem:differencePivotal}.
	
Let $\{v_i\}_i$ be a homogeneous weight basis of $V$ with $v_i$ of weight $\mu_i \in \Cart^{\vee}$. Let $w_0 \in W$ be a highest-weight vector. Write $S^{\prime}(V,W) = X_1 \circ X_2 \circ X_3 \circ X_4$, where $X_j$ corresponds to the obvious factor in equation \eqref{eqn:partialtrace}. We compute
\[
X_3 X_4 w_0
=
\sum_i (-1)^{\p{w}_0 \p{v}_i}q^{- \langle \lambda,\mu_i \rangle} v_i \otimes  w_0 \otimes v_i^{\vee}.
\]
Applying $X_2$ to this expression gives
\[
\sum_i q^{- \langle 2 \lambda,\mu_i \rangle} w_0 \otimes  v_i \otimes v_i^{\vee} + \sum_l w_l^{\prime} \otimes v_l^{\prime} \otimes v_l^{\vee},
\]
where the first and second sums arise from the application of $\tau \circ \mathcal{H}^q$ and $\tau \circ (\check{\mathcal{R}}^q-1) \mathcal{H}^q$, respectively. The explicit form of $\check{\mathcal{R}}^q$ shows that $v_l^{\prime}$ is a linear combination of basis vectors of $V$ of weights strictly greater than $\mu_l$. Applying $X_1$ to this expression then gives
\[
\sum_i (-1)^{\p{v}_i} q^{- \langle 2\lambda - \pi,\mu_i \rangle} w_0.
\]
Applying $\phi_{-2 \lambda + \pi}$ to the equality $\sch(V) = \sum_i (-1)^{\overline{v}_i} e^{\mu_i}$ gives the same expression, thereby proving the first statement.
 
For the second statement, we use Example \ref{ex:superChar} to compute
\[
\phi_{-2 \lambda + \pi} ( \sch(K(\mu)) )
=
q^{-\langle 2\lambda - \pi, \mu \rangle} \prod_{\alpha \in \Delta_{\p 0}^+} \left( \sum_{j=0}^{r-1} q^{\langle 2\lambda - \pi, j \alpha \rangle} \right)  \prod_{\alpha \in \Delta_{\p 1}^+} \left( 1 - q^{\langle 2 \lambda - \pi, \alpha \rangle} \right).
\]
Evaluating the geometric sum and simplifying gives the claimed equality.
\end{proof}

\begin{Cor}\label{cor:mtraceNormalized}
    The modified trace $\mt$ can be normalized so that the modified dimension of a generic Kac module is
\begin{equation}
\label{eq:modDimKac}
\qd(K(\lambda))
=
\frac{\prod_{\alpha \in \Delta^+} \left\{ \langle\lambda - \frac{\pi}{2},\alpha \rangle \right\}^{\epsilon(\alpha)}_q }{ \prod_{\alpha \in \Delta_{\p 0}^+} \left\{ {r \langle \lambda - \frac{\pi}{2},\alpha \rangle} \right\}_q }.
\end{equation}
\end{Cor}
\begin{proof}
Let $V, W \in \cat$ be simple projectives. Isotopy invariance of $\mt$ implies the identity
\[
\qd(V) \langle S^{\prime}(W,V) \rangle
=
\qd(W) \langle S^{\prime}(V,W) \rangle.
\]
Applying this to $V = K(\lambda)$ and $W = K(\mu)$ for generic $\lambda, \mu \in \Cart^{\vee}$ and using the explicit form of $S^{\prime}$ from Lemma \ref{lem:longHopf}---in particular, the symmetry of the overall $q$-exponential factor under the exchange of $\lambda$ and $\mu$---we see that $\mt$ can be normalized as stated.
\end{proof}

For later use, we combine Lemma \ref{lem:longHopf} and Corollary \ref{cor:mtraceNormalized} to write
\begin{equation}
\label{eq:longHopfDim}
\langle S^{\prime}(K(\lambda),K(\mu)) \rangle
=
\frac{q^{-2 \langle \lambda - \frac{\pi}{2},\mu - \frac{\pi}{2} \rangle}}{\qd(K(\mu))}.
\end{equation}

\begin{Cor}
\label{cor:dimVrho}
The module $K((r - 1)\rho_{\p 0}) \in \cat_{[0]}$ is a simple projective of dimension $r^{n^2}2^{2n}$. 
\end{Cor}
\begin{proof}

The degree statement follows from the assumption that $r$ is odd and $2 \rho_{\p 0} \in \Lambda_R$.

By \cite[\S 2.2]{derenzi2020}, $(r - 1)\rho_{\p 0}$ is the highest weight of a simple Verma $\overline{U}_q^H(\ospn_{\p 0})$-module. In particular, $K((r - 1)\rho_{\p 0})$ has the stated dimension.

We use Proposition \ref{prop:kacTypical} to prove that $(r - 1)\rho_{\p 0}$ is typical.
For $\alpha \in \Delta_{\p 1}^+$, we have
\[
\langle (r - 1)\rho_{\p 0} + \rho , \alpha \rangle
=
r\langle \rho_{\p 0}, \alpha \rangle - \langle\rho_{\p 1} ,\alpha \rangle.
\]
Parts \eqref{ite:oddRhoOdd} and \eqref{ite:eveRhoOdd} of Lemma \ref{lem:rhoPairings} imply that $\langle\rho_{\p 1} ,\alpha \rangle = -n$ and $r\langle \rho_{\p 0}, \alpha \rangle \in r\Z$, respectively. It follows that $(r-1) \rho_{\p 0}$ is typical if and only if $q^{2n} \neq 1$, that is, $r$ does not divide $n$.

It remains to prove projectivity. Let $\lambda \in \Cart^{\vee}$ be generic. Lemma \ref{lem:longHopf} shows that $S^{\prime}(K(\lambda),K((r - 1)\rho_{\p 0}))$ is non-zero. On the other hand, $S^{\prime}(K(\lambda),K((r - 1)\rho_{\p 0}))$ is the composition
\[
K((r - 1)\rho_{\p 0})
\xrightarrow[]{i} 
K(\lambda) \otimes K((r - 1)\rho_{\p 0}) \otimes K(\lambda)^{\vee}
\xrightarrow[]{\pi}
K((r - 1)\rho_{\p 0}),
\]
where $i = X_1 \circ X_2$ and $\pi = X_3 \circ X_4$; see the proof of Lemma \ref{lem:longHopf}. After normalizing $i$ and $\pi$, this realizes $K((r-1)\rho_{\p 0})$ as a retract of $K(\lambda) \otimes K((r - 1)\rho_{\p 0}) \otimes K(\lambda)^{\vee}$. Projectivity of $K((r - 1)\rho_{\p 0})$ then follows from that of $K(\lambda)$.
\end{proof}

\subsection{Relative modularity}
\label{sec:relMod}

We use the notation of Example \ref{ex:oneDimModules} and Definition \ref{def:freeRealWeights}. Let $\FR$  be the abelian group $\Lambda_{\FR} \times \Ztwo$. Write an element $z \in \FR$ as $(\lambda_z, \p{z})$.

\begin{Lem}
\label{lem:freeRealization}
The assignment $z \mapsto \sigma(z)$ defines a free realization of $\FR$ on $\cat_{[0]}$.
\end{Lem}

\begin{proof}
Proposition \ref{prop:ribCat} gives $\theta_{\sigma(z)} = q^{- \langle \lambda_z - \pi,\lambda_z \rangle} \id_{\sigma(z)}$. The proof of Lemma \ref{lem:alternativeZ} shows that $q^{- \langle \lambda_z - \pi,\lambda_z \rangle} = 1$, as required. A direct computation gives $\qdim_{\cat} \, \sigma(z) = (-1)^{\p{z}} q^{-\langle \pi, \lambda_z \rangle}$ which, again by the proof of Lemma \ref{lem:alternativeZ}, is equal to $(-1)^{\p{z}}$. The remaining properties of a free realization are clear.
\end{proof}

\begin{Lem}
\label{lem:bicharacter}
The bilinear map $\psi: \Gr \times \FR \rightarrow \C^{\times}$, $([ \lambda ],z) \mapsto  q^{-2 \langle \lambda, \lambda_z \rangle}$, satisfies the conditions of Definition \ref{def:preMod}.
\end{Lem}

\begin{proof}
Let $V \in \cat_{[\lambda]}$ with $v \in V$ of weight $\mu$ and $w \in \sigma(z)$. A direct computation gives
\begin{equation*}
c_{\sigma(z),V} \circ c_{V, \sigma(z)}(v \otimes w) = q^{-2 \langle \mu, \lambda_z \rangle} v \otimes w.
\end{equation*}
Since $\lambda - \mu \in \Lambda_R$ and $2 \langle \Lambda_R, \lambda_z \rangle \in r \Z$ (see the proof of Lemma \ref{lem:alternativeZ}), we have $q^{2 \langle \mu, \lambda_z \rangle} = q^{2 \langle \lambda ,\lambda_z \rangle}$, as required.
\end{proof}

\begin{Def}
\label{def:transparent}
Let $U \in \cat_{[0]}$. A morphism $f \in \End_{\cat}(U)$ is \emph{transparent} if, for any $V,W \in \cat_{[0]}$, there are equalities
\[
c_{U,V} \circ (f \otimes \id_V) \circ c_{V,U} = \id_V \otimes f,
\qquad
c_{W,U} \circ (\id_W \otimes f) \circ c_{U,W} = f \otimes \id_W.
\]
\end{Def}

The next result is a direct analogue of \cite[Lem. 2.3]{derenzi2020}.

\begin{Lem}
\label{lem:transparentFactor}
    Let $V \in \cat_{[0]}$ with transparent morphism $f \in \End_{\cat}(V)$. There exist $z_i \in \FR$ and $\pi_i \in \Hom_{\cat}(V,\sigma(z_i))$ and $\iota_i \in \Hom_{\cat}(\sigma(z_i),V)$ such that $f = \sum_{i = 1}^m \iota_i \circ \pi_i$.
\end{Lem}

\begin{proof}
Let $w_0$ be a highest-weight vector of the simple module $K((r - 1) \rho_{\p 0}) \in \cat_{[0]}$ of Corollary \ref{cor:dimVrho}
and $v \in V$ of weight $\lambda$. The explicit description of the braiding shows that $c_{K((r - 1) \rho_{\p 0}),V}(w_0 \otimes v)$ is a non-zero scalar multiple of $v \otimes w_0$. By transparency, $c_{V,K((r - 1) \rho_{\p 0})}(f(v) \otimes w_0)$ is then a non-zero scalar multiple of $w_0 \otimes f(v)$. Using the weight basis
\[
\Big\{ \prod_{\alpha \in \Delta_+} F_{\alpha}^{y_{\alpha}}  \cdot w_0 \mid 0 \leq y_{\alpha} \leq m_{\p \alpha} \Big\}
\]
of $K((r - 1) \rho_{\p 0})$, we conclude that $\prod_{\alpha \in \Delta_+} E_{\alpha}^{y_{\alpha}} \cdot f(v) = 0$ for all $y_{\alpha}$ as above (not all zero). Using instead a lowest-weight vector of $K((r - 1) \rho_{\p 0})$ and the second of the transparency conditions gives $\prod_{\alpha \in \Delta_+}F_{\alpha}^{y_{\alpha}} \cdot f(v) = 0$ for all $y_{\alpha}$ as above (not all zero). We conclude that $[E_i,F_i] f(v) = 0$, that is, $2 \langle \lambda,\alpha_i \rangle \in r \Z$ for all $0 \leq i \leq n$. It follows from Example \ref{ex:oneDimModules} and the proof of Lemma \ref{lem:alternativeZ} that $f(v)$ generates a one-dimensional module of the form $\sigma(z)$, $z \in \FR$. Applying this argument to a weight basis of $V$ completes the proof.
\end{proof}

\begin{Thm}
\label{thm:modularity}
The category $\cat$ is relative modular with $\zeta = r^{n + 1}$.
\end{Thm}

\begin{proof}
It remains to establish the existence of a relative modularity parameter $\zeta$. Consider Definition \ref{def:modG} with $h = [\gamma]$ and $g = [\lambda]$ with $V_i=K(\alpha)$ and $V_j = K(\beta)$. The morphism $f_{[\gamma]; \alpha, \beta} \in \End_{\cat}(K(\alpha) \otimes K(\beta)^{\vee})$ determined by the left-hand side of diagram \eqref{eq:mod} is transparent in $\cat_{[0]}$; see \cite[Lem. 5.9]{costantino2014}. By Lemma \ref{lem:transparentFactor}, we can write
\[
f_{[\gamma]; \alpha, \beta}
=
\sum_{i=1}^m h_{[\gamma]; \alpha, \beta, i} \circ g_{[\gamma]; \alpha, \beta, i}
\]
for some
\[
g_{[\gamma]; \alpha, \beta, i} \in \Hom_{\cat}(K(\alpha) \otimes K(\beta)^{\vee},\sigma(z_i)),
\qquad
h_{[\gamma]; \alpha, \beta, i} \in \Hom_{\cat}(\sigma(z_i) ,K(\alpha) \otimes K(\beta)^{\vee}).
\]
Noting that
\[
\Hom_{\cat}(K(\alpha) \otimes K(\beta)^{\vee}, \sigma(z_i)) \simeq \Hom_{\cat}(K(\alpha), K(\beta) \otimes \sigma(z_i)),
\]
we conclude that $z_i =0 $ and $\alpha = \beta$ by generic semisimplicity. We may therefore assume that $f_{[\gamma]; \alpha, \beta}$ is a scalar multiple of $\tcoev_{K(\alpha)} \circ \ev_{K(\alpha)}$. To compute the scalar, we compare the modified traces of $\tcoev_{K(\alpha)} \circ \ev_{K(\alpha)}$ and $f_{[\gamma]; \alpha, \alpha}$. The former is $\qd(K(\alpha))$. To compute the latter, write
\[
f_{[\gamma]; \alpha, \alpha} = \sum_{\mu \in \gamma +  I} \qd(K(\mu))f_{\mu,\alpha,\alpha},
\]
where $I \subset \Lambda_R$ is a choice of representatives of $\Lambda_R \slash \Lambda_{\FR}$. We compute
\begin{eqnarray*}
\mt_{K(\alpha) \otimes K(\alpha)^{\vee}}(f_{[ \gamma ],\alpha, \alpha})
&=&
\sum_{\mu \in \gamma + I}\qd(K(\mu)) \mt_{K(\alpha) \otimes K(\alpha)^{\vee}}(f_{\mu,\alpha, \alpha})\\
&=&
\sum_{\mu \in \gamma + I} \qd(K(\mu)) \qd(K(\alpha))^2 \langle S^{\prime}(K(\mu), K(\alpha)^{\vee}) \rangle \langle S^{\prime}(K(\alpha), K(\mu)) \rangle \\
&=&
\sum_{\mu \in \gamma + I} \qd(K(\alpha))
=
\qd(K(\alpha)) r^{n+1}.
\end{eqnarray*}
The second equality follows from isotopy invariance and defining properties of $\mt$. The third equality follows from Lemma \ref{lem:dualKac} and equation \eqref{eq:longHopfDim}. We conclude that $\zeta = r^{n + 1}$.
\end{proof}

\subsection{Stabilization coefficients}
\label{sec:stabCoeff}

We compute the stabilization coefficients $\Delta_{\pm}$ from Definition \ref{def:ndeg}.

\begin{Lem}
\label{lem:invt}
Let $\alpha \in \Delta^+$. The function $\Cart^{\vee} \rightarrow \C$, $\lambda \mapsto \left\{r\langle\lambda - \frac{\pi}{2},\alpha \rangle \right\}_q$, descends to a function $\Gr \rightarrow \C$.
\end{Lem}

\begin{proof}
Since $\langle-,-\rangle$ is an integral form on $\Lambda_R$ and $q$ is an $r$\textsuperscript{th} root of unity, we have $q^{-2 r \langle k, \alpha \rangle} = 1$ for all $k \in \Lambda_R$ and $\alpha \in \Delta^+$. The lemma follows.
\end{proof}

Let $b \in \{ \pm \}$. For $V \in \cat$, introduce the notation $V^+ = V$ and $V^- = V^{\vee}$. Then the stabilization coefficients are
\begin{multline*}
\Delta_b
=
\sum_{k \in I} \qd(K(\lambda+k)) \langle \theta^b_{K(\lambda-r+1)} \rangle \langle \theta^b_{K(\lambda+k)^{\vee}} \rangle \langle S^{\prime}(K(\lambda+k)^{-b},K(\lambda-r+1)) \rangle \\
=
q^{\frac{b}{2} \langle \pi, \pi \rangle} \sum_{k \in I} q^{-b \langle k, k \rangle} \frac{\mathsf{d}(K(\lambda + k))}{\mathsf{d}(K(\lambda))}
=
q^{\frac{b}{2} \langle \pi, \pi \rangle} \sum_{k \in I} q^{-b \langle k, k \rangle} \prod_{\alpha \in \Delta^+} \frac{\left\{ \langle\lambda + k - \frac{\pi}{2},\alpha \rangle \right\}^{\epsilon(\alpha)}_q }{\left\{ \langle\lambda - \frac{\pi}{2},\alpha \rangle \right\}^{\epsilon(\alpha)}_q}.
\end{multline*}
The first equality follows from the definition of $\Delta_b$, the second from Proposition \ref{prop:ribCat} and Lemma \ref{lem:longHopf} and the third from equation \eqref{eq:modDimKac} and Lemma \ref{lem:invt}. By \cite[Lem. 5.10]{costantino2014}, $\Delta_b$ is independent of the generic weight $\lambda$ used in its computation. With this in mind, for each $t \in \R$ set $\lambda_t = -\I t \sum_{i=0}^n d_i w_i \in \Cart^{\vee}$. This choice of $\lambda_t$ ensures that $\lim_{t \rightarrow \infty} q^{\langle \lambda_t, \alpha \rangle} = \infty$ for all $\alpha \in \Delta^+$. Using this, we compute
\[
\Delta_b
=
\lim_{t \rightarrow \infty} q^{\frac{b}{2} \langle \pi, \pi \rangle} \sum_{k \in I} q^{-b \langle k, k \rangle} \prod_{\alpha \in \Delta^+} \frac{\left\{ \langle \lambda_t + k - \frac{\pi}{2},\alpha \rangle \right\}^{\epsilon(\alpha)}_q }{\left\{ \langle \lambda_t - \frac{\pi}{2},\alpha \rangle \right\}_q^{\epsilon(\alpha)}} \\
=
q^{\frac{b}{2} \langle \pi, \pi \rangle} \sum_{k \in I} q^{-b \langle k, k \rangle +\langle k,2 \rho \rangle}.
\]

\begin{Lem}
\label{lem:stabCompConj}
The stabilization coefficient $\Delta_+$ is the complex conjugate of $\Delta_-$.
\end{Lem}

\begin{proof}
This follows from a short calculation using the above formula for $\Delta_b$.
\end{proof}


In view of Lemma \ref{lem:stabCompConj}, we restrict attention to $\Delta_+$. 
Completing the square gives
\[
\Delta_+ = q^{\frac{1}{2} \langle \pi, \pi \rangle + \langle \rho, \rho \rangle} \sum_{k \in I} q^{-\langle k, k\rangle}.
\]
Given an (ortho)symplectic Lie (super)algebra $\mathfrak{g}$ with Cartan subalgebra $\Cart$, define $I_{\mathfrak{g}} = \Lambda_R \slash r \Lambda_R$ and let $\langle-,-\rangle_{\mathfrak{g}}$ be the symmetric bilinear form on $\Cart^{\vee}$ defined by the symmetrized Cartan matrix. We use the decomposition $I_{\ospn} \simeq \Z \slash r \Z \cdot \alpha_0 \times I_{\mathfrak{sp}(2n)}$ to compute
\[
\sum_{k \in I_{\ospn}} q^{-\langle k, k\rangle_{\ospn}}
=
\sum_{l=0}^{r-1} \sum_{k \in I_{\mathfrak{sp}(2n)}} q^{-\langle l \alpha_0 + k, l \alpha_0 + k \rangle_{\ospn}}
=
\sum_{k \in I_{\mathfrak{sp}(2n)}} \left( \sum_{l=0}^{r-1} q^{2l k_1} \right) q^{-\langle k, k \rangle_{\mathfrak{sp}(2n)}},
\]
the second equality following from the explicit form of the symmetrized Cartan matrix of $\ospn$. Since $r$ is odd, we have $\sum_{l=0}^{r-1} q^{2l k_1} = r \delta_{k_1,0}$ which gives
\[
\sum_{k \in I_{\ospn}} q^{-\langle k, k\rangle_{\ospn}}
=
r \sum_{k \in I_{\mathfrak{sp}(2n)}} \delta_{k_1,0} q^{-\langle k, k \rangle_{\mathfrak{sp}(2n)}}
=
r \sum_{k \in I_{\mathfrak{sp}(2n-2)}} q^{-\langle k, k \rangle_{\mathfrak{sp}(2n-2)}}.
\]
The second equality follows from the observation that deleting the first row and column of the symmetrized Cartan matrix of $\mathfrak{sp}(2n)$---corresponding to setting $k_1=0$---gives that of $\mathfrak{sp}(2n-2)$. Since $\langle -, - \rangle_{\mathfrak{sp}(2n-2)}$ is negative definite, the final sum can be computed using standard quadratic Gauss sum techniques; see \cite{berndt1998}. We find
\[
\sum_{k \in I_{\mathfrak{sp}(2n)}} q^{-\langle k, k \rangle_{\mathfrak{sp}(2n)}}
=
\epsilon(r,n) r^{\frac{n}{2}},
\]
where
\[
\epsilon(r,n)
=
\begin{cases}
1 & \mbox{if } r \equiv 1 \mod 4, \\
i^{-n} & \mbox{if } r \equiv 3 \mod 4.
\end{cases}
\]
Putting this all together, we have
\[
\Delta_+
=
q^{\frac{1}{2} \langle \pi, \pi \rangle + \langle \rho, \rho \rangle} \epsilon(r,n-1) r^{\frac{n-1}{2}}.
\]
Direct computations, together with Lemma \ref{lem:rhoPairings}\eqref{ite:oddRhoSimp}, give
\[
\frac{1}{2} \langle \pi, \pi \rangle + \langle \rho, \rho \rangle
\equiv
\frac{n(n-1)(2n-1)}{2}
\mod r.
\]


\newcommand{\etalchar}[1]{$^{#1}$}
\providecommand{\bysame}{\leavevmode\hbox to3em{\hrulefill}\thinspace}
\providecommand{\MR}{\relax\ifhmode\unskip\space\fi MR }
\providecommand{\MRhref}[2]{%
  \href{http://www.ams.org/mathscinet-getitem?mr=#1}{#2}
}
\providecommand{\href}[2]{#2}

\end{document}